\pdfoutput=1

\documentclass[reqno]{amsart}

\usepackage{lmodern}
\usepackage[T1]{fontenc}
\usepackage[utf8]{inputenc}
\usepackage[english]{babel}
\usepackage{microtype} 

\usepackage{amsmath,amssymb,amsthm,mathrsfs,
latexsym,mathtools,mathdots,booktabs,xparse,appendix,enumerate,enumitem,tikz,bm,url}
\usepackage[centertableaux]{ytableau}

\usetikzlibrary{arrows,positioning,shapes.geometric}
\usepackage[pdftex]{hyperref}
\usepackage[capitalize,noabbrev]{cleveref}

\usepackage{todonotes}
\presetkeys%
{todonotes}%
{inline,backgroundcolor=yellow}{}

\newtheorem{theorem}{Theorem}
\newtheorem{proposition}[theorem]{Proposition}
\newtheorem{lemma}[theorem]{Lemma}
\newtheorem{corollary}[theorem]{Corollary}

\theoremstyle{definition}
\newtheorem{definition}[theorem]{Definition}
\newtheorem{example}[theorem]{Example}
\newtheorem{remark}[theorem]{Remark}

\definecolor{lightblue}{rgb}{0.8,0.8,1.0}
\definecolor{lightgreen}{rgb}{0.8,1.0,0.8}
\definecolor{pBlue}{RGB}{86,139,190}
\definecolor{pCyan}{RGB}{149,186,201}
\definecolor{pSand}{RGB}{184,166,121}
\definecolor{pAlgae}{RGB}{87,115,135}
\definecolor{pSkin}{RGB}{236,216,167}
\definecolor{pGray}{RGB}{156,175,156}
\definecolor{pPink}{RGB}{215,114,127}
\definecolor{pOrange}{RGB}{211,153,80}

\newcommand{\defin}[1]{%
\relax\ifmmode%
\textcolor{blue}{#1}%
\else\textcolor{blue}{\emph{#1}}%
\fi%
}

\newcommand{\thsup}{\textnormal{th}}

\newcommand{\setN}{\mathbb{N}}
\newcommand{\setZ}{\mathbb{Z}}
\newcommand{\setQ}{\mathbb{Q}}
\newcommand{\setC}{\mathbb{C}}

\newcommand{\avec}{\mathbf{a}}

\newcommand{\symS}{S}

\DeclareMathOperator{\rev}{rev}
\DeclareMathOperator{\flip}{flip}
\DeclareMathOperator{\revflip}{revflip}

\newcommand{\PAP}{\mathrm{P}}
\newcommand{\RRP}{\mathrm{MP}}
\newcommand{\rrp}{\mathrm{\mathfrak{mp}}}
\newcommand{\pap}{\mathrm{\mathfrak{p}}}

\newcommand{\SEF}{\mathcal{F}}  
\DeclareMathOperator{\sefToPerm}{\mathtt{sefToPerm}}

\tikzset{every picture/.append
  style={scale=1,
	baseline=(current bounding box.center),
	x=1em,
	y=1em,
	thinLine/.style={line width=0.7pt},
	thickLine/.style={line width=1.4pt,line join=round},
	snaking/.style={decorate,decoration={zigzag,segment length=1.0mm,amplitude=0.15mm}},
	fillGrey/.style={fill=black!30},
	entries/.style={xshift=-0.5em,yshift=-0.5em,font=\small},
	bgEntry/.style={xshift=-0.5em,yshift=-0.5em,font=\small,
		regular polygon,regular polygon sides=4,fill,inner sep=0pt,minimum size=1em
	},
	circled/.style={circle,draw,inner sep=0pt,minimum size=1em},
	mydot/.style={circle,fill=black,inner sep=0pt,minimum size=1mm}
	}
}

\title{Pattern-avoidance and Fuss--Catalan numbers}

\author[P. Alexandersson]{Per Alexandersson}
\address{Department of Mathematics, Stockholm University, Sweden}
\email{per.w.alexandersson@gmail.com}

\author[S.A. Fufa]{Samuel Asefa Fufa}
\address{Department of Mathematics, Addis Ababa University, Ethiopia}
\email{samuel.asefa@aau.edu.et}

\author[F. Getachew]{Frether Getachew}
\address{Department of Mathematics, Addis Ababa University, Ethiopia}
\email{frigetach@gmail.com}

\author[D. Qiu]{Dun Qiu}
\address{Mathematics Department, Beijing Jiaotong University, P.~R.~China}
\email{qiudun@bjtu.edu.cn}

\begin{document}

\begin{abstract}
We study a subset of permutations, where entries are restricted to having the same 
remainder as the index, modulo some integer $k \geq 2$.
We show that when also imposing the classical 132- or 213-avoidance restriction on the permutations, 
we recover the Fuss--Catalan numbers and some special cases of the Raney numbers.

Surprisingly, an analogous statement also holds when we 
impose the mod $k$ restriction on a Catalan family of subexcedant functions.

Finally, we completely enumerate all combinations of mod-$k$-alternating permutations,
avoiding two patterns of length 3. This is analogous to the systematic study
by Simion and Schmidt, of permutations avoiding two patterns of length 3.
\end{abstract}

\maketitle

\section{Introduction}

It is classically known that the number of permutations of length $n$
avoiding a permutation-pattern of length $3$, is counted by the Catalan number, $C_n$.
There is a plethora of other combinatorial objects which are counted by $C_n$,
see e.g., \cite{StanleyCatalan}.
One classical Catalan family are the \emph{Dyck paths}---paths using
steps in $\{(1,0),(0,1)\}$, starting at $(0,0)$, ending at $(n,n)$
and never going below the line $y=x$. This latter family can be generalized,
to instead counting paths never going below $y=kx$,
and ending at $(kn,n)$ for some fixed positive integer
$k$. 
The number of such paths are enumerated by the \emph{Fuss--Catalan numbers}, $C_n^{k+1}$.
In this paper, we find several \emph{Fuss}-generalizations of pattern-avoiding permutations,
by imposing additional restrictions on the entries in the permutations.

A parity-alternating permutation starting with an odd integer (PAP),
sends even integers to even integers and odd integers to odd. 
This set of permutations have been studied previously in
\cite{Tanimoto2010} when considering the ascent statistic and signed Eulerian numbers,
and later in \cite{KebedeRakotondrajao2021}, where parity-alternating dearrangements
are enumerated.
We generalize the notion of parity-alternating permutations,
to \emph{mod-$k$-alternating permutations}. Here, we require
that $\pi_i$ is congruent to $i$ mod $k$, where $k \geq 1$ is a fixed integer.
This notion is not to be confused with the generalization in
\cite{Munagi2014}, where the size of blocks of entries with
same parity is restricted (but not restricted to be $1$).

We enumerate the mod-$k$-alternating permutations under pattern-avoidance,
for each of the patterns in $\{132,213,231,312\}$.
Moreover, we also enumerate the mod-$k$-alternating
permutations which avoids two patterns of length $3$.
The number of mod-$k$-alternating permutations of length $n$
avoiding $\sigma$ is denoted $\rrp_\sigma(n,k)$.

We also consider subexcedant functions.
These are simply words of positive integers, $f_1,f_2,\dotsc,f_n$,
such that $f_i \leq i$. There is a (specific) bijection between 
permutations of length $n$ and subexcedant functions of length $n$ 
(introduced in \cite{MantaciRakotondrajao2001}) so one can study 
subexcedant functions instead of permutations.
Under this bijection, the mod-$k$-alternating restriction translates to 
the exact same restriction on the subexedant function.
In other words, $f_i \equiv i$ mod $k$, see \cref{prop:sefRRPs}.

\subsection{Main results}

In \cref{sec:modKPattern}, we prove the following theorem.
In particular, this shows that $\rrp_\sigma(n,k)$ is a Raney 
number and that $\rrp_\sigma(km,k)$ is a Fuss--Catalan number.
\begin{theorem}[{\cref{cor:mainFussCatalan} below}]
Let $\sigma \in \{132,213\}$ and $n,k \geq 1$.
Let $m$ and $j$ be defined via $n=km+j$ with $0 \leq j < k$.
Then
\[
\rrp_\sigma(n,k) =  \frac{j+1}{km+j+1} \binom{ (k+1)m + j }{km+j}
\]
and in particular
\[
\rrp_\sigma(km,k) = \frac{1}{(k+1)m+1} \binom{(k+1)m+1}{m}
	= \frac{1}{km+1} \binom{(k+1)m}{m}.
\]
\end{theorem}

In \cref{sec:subex}, we prove our second main result.
We let \defin{$\rrp_C(n,k)$} denote the set of mod-$k$-alternating
permutations whose subexcedant function satisfy the following two properties:
\begin{itemize}
 \item $f_i \equiv i$ mod $k$,
 \item $1 \leq f_i \leq f_{i-1}+1$ for all $i \in \{2,3,\dotsc,n\}$.
\end{itemize}
The first property is analogous to the mod-$k$-alternating property 
on permutations. Subexcedant functions which satisfy the second property
are enumerated by the Catalan numbers; the second condition is therefore
a natural analog to permutations avoiding a fixed length-3 pattern.
We have the following result:
\begin{theorem}[{\cref{cor:subexFuss} below}]
Let $m$ and $j$ be defined via $n=km+j$ with $0 \leq j < k$.
Then
\[
\rrp_C(n,k) =  \frac{j+1}{km+j+1} \binom{ (k+1)m + j }{km+j}
\]
and in particular
\[
\rrp_C(km,k) = \frac{1}{(k+1)m+1} \binom{(k+1)m+1}{m}
	= \frac{1}{km+1} \binom{(k+1)m}{m}.
\]
\end{theorem}

Finally, in \cref{sec:2patterns} we systematically enumerate 
all families of mod-$k$-alternating permutations 
avoiding two (different) patterns of length 3.

\section{Preliminaries}

Given a word $w = [w_1,\dotsc,w_\ell]$ and an integer $n$, we let $n+w$ denote
the word $[w_1+n,w_2+2,\dotsc,w_\ell+n]$. This convention is extended
to the expressions $w-n$ and $n-w$ in the natural manner.
If $\alpha = [\alpha_1,\dotsc,\alpha_r]$, $\beta = [\beta_1,\dotsc,\beta_s]$
are words, then $[\alpha,\beta]$ denotes the concatenation
$
 [\alpha_1,\dotsc,\alpha_r, \beta_1,\dotsc,\beta_s]
$.
We use several variants of this convention. 
We use the term \defin{subword} to denote a sequence 
of consecutive entries in a word.

The set $\{1,2,\dotsc,n\}$ is denoted by $\defin{[n]}$ 
(the context should make it clear that this is not
to be interpreted as the word of one letter).
The set of permutations of $[n]$ is denoted by $\symS_n$.
We use brackets for permutations in one-line notation and regular parenthesis to denote cycles.
We make a use of four well-known involutions on $\symS_n$,
\emph{inverse}, \emph{reverse}, \emph{flip} and \emph{revflip};
for $\pi = [\pi_1,\dotsc,\pi_n] \in \symS_n$, we let
\[
\defin{\rev(\pi)} \coloneqq [\pi_n,\pi_{n-1},\dotsc,\pi_1], \;
\defin{\flip(\pi)} \coloneqq n+1-\pi \;
\text{ and } \;
\defin{\revflip(\pi)} \coloneqq \flip(\rev(\pi)).
\]

For permutations $\pi,\omega \in \symS_n$, we think
of multiplication as composition of functions, so that
$(\pi \circ \omega)(k) = \pi(\omega(k))$ whenever $1 \leq k \leq n$.
Given $\pi$, we associate its \defin{permutation matrix},
\[
\defin{M_\pi} \coloneqq \left( \delta_{i, \pi(j)} \right)_{1 \leq i, j \leq n}.
\]
The permutation matrix of the composition $\pi \circ \omega$ is $M_{\pi} M_{\sigma}$.
We follow the convention in \cite{Kitaev2011}
and illustrate permutation matrices by letting row indices start at the \emph{bottom},
as in \cref{fig:permMatrix}. From now on, this is the picture we have in mind when
referring to permutation matrices.
\begin{figure}[!ht]
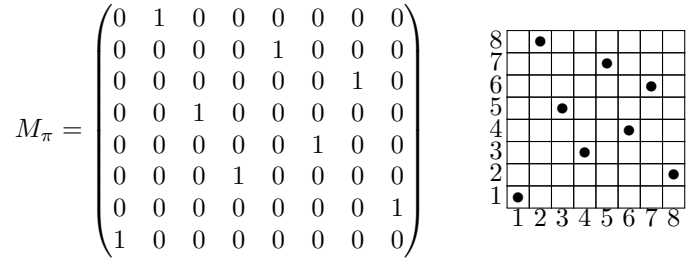

	\[
	M_\pi =
	\begin{pmatrix}
	0 & 1 & 0 & 0 & 0 & 0 & 0 & 0 \\
	0 & 0 & 0 & 0 & 1 & 0 & 0 & 0 \\
	0 & 0 & 0 & 0 & 0 & 0 & 1 & 0 \\
	0 & 0 & 1 & 0 & 0 & 0 & 0 & 0 \\
	0 & 0 & 0 & 0 & 0 & 1 & 0 & 0 \\
	0 & 0 & 0 & 1 & 0 & 0 & 0 & 0 \\
	0 & 0 & 0 & 0 & 0 & 0 & 0 & 1 \\
	1 & 0 & 0 & 0 & 0 & 0 & 0 & 0
	\end{pmatrix}
	\qquad
\ytableausetup{boxsize=0.8em}
	\begin{ytableau}
	\none[8]& \, & \bullet & \, & \, & \, & \, & \, & \, \\
	\none[7]& \, & \, & \, & \, & \bullet & \, & \, & \, \\
	\none[6]&\, & \, & \, & \, & \, & \, & \bullet & \, \\
	\none[5]& \, & \, & \bullet & \, & \, & \, & \, & \, \\
	\none[4]& \, & \, & \, & \, & \, & \bullet & \, & \, \\
	\none[3]& \, & \, & \, & \bullet & \, & \, & \, & \, \\
	\none[2]& \, & \, & \, & \, & \, & \, & \, & \bullet \\
	\none[1]& \bullet & \, & \, & \, & \, & \, & \, & \, \\
	\none   & \none[1]&\none[2]&\none[3]&\none[4]&\none[5]&\none[6]&\none[7]&\none[8]
	\end{ytableau}
	\]
	\caption{
		The permutation matrix $M_\pi$ associated with $\pi = [1, 8, 5, 3, 7, 4, 6, 2]$,
		where the bottom row is \emph{the first row}.
		The rightmost figure is simply the graph of the function $i \mapsto \pi(i)$
		in Cartesian coordinates.
	}\label{fig:permMatrix}
\end{figure}

\subsection{Permutation patterns}

A permutation $\pi \in \symS_n$ is said to \defin{contain} the pattern $\sigma \in \symS_k$,
if there is a subsequence of $\pi$ which is order-isomorphic to $\sigma$.
Otherwise, $\pi$ is said to avoid $\sigma$ and $\defin{\symS_{\sigma}(n)}$ denotes
the set of $\sigma$-avoiding permutations over $[n]$. We simply write a pattern as a word.
For example, $\pi = [1, 8, 5, 3, 7, 4, 6, 2]$ contains the pattern $231$
as the subsequence $3, 6, 2$ in $\pi$ have its elements in the same relative order.
However, $[1, 7, 6, 4, 2, 3, 5, 8]$ avoids the pattern $231$.

\begin{figure}[!ht]
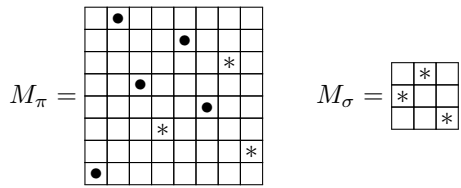

	\[
	M_\pi =
	\ytableausetup{boxsize=0.8em}
	\begin{ytableau}
	\, & \bullet & \, & \, & \, & \, & \, & \, \\
	\, & \, & \, & \, & \bullet & \, & \, & \, \\
	\, & \, & \, & \, & \, & \, & \ast & \, \\
	\, & \, & \bullet & \, & \, & \, & \, & \, \\
	\, & \, & \, & \, & \, & \bullet & \, & \, \\
	\, & \, & \, & \ast & \, & \, & \, & \, \\
	\, & \, & \, & \, & \, & \, & \, & \ast \\
	\bullet & \, & \, & \, & \, & \, & \, & \, \\
	\end{ytableau}
	\qquad
	M_\sigma =
	\begin{ytableau}
	\, & \ast & \, \\
	\ast & \, & \, \\
	\, & \, & \ast \\
	\end{ytableau}
	\]
	\caption{
		The permutation matrices associated with $\pi = [1, 8, 5, 3, 7, 4, 6, 2]$
		and $\sigma = 231$. Here we see that $\pi$
		contains the pattern $\sigma$ (there are other instances of this pattern in $\pi$).
	}\label{fig:permPatternContainment}
\end{figure}

Patterns of length $3$ is the first non-trivial case to consider and we use a one-line notation without the parenthesis to represent them.
In \cref{fig:132Avoiding}, we show the structure of $132$-avoiding permutations.
Any permutation of length $n$, avoiding the pattern $132$, must be of the form
\[
  \pi = [\alpha_1, \dotsc, \alpha_j, \; n, \; \beta_1,\dotsc,\beta_{n-1-j}],
\]
where $\min(\alpha_1,\dotsc,\alpha_j) > \max(\beta_1,\dotsc,\beta_{n-1-j})$.
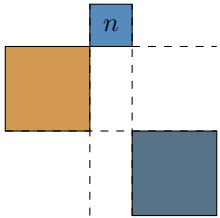
\begin{figure}[!ht]
	\centering
	\begin{tikzpicture}[scale=1.6]
	\draw[black,fill=pOrange] (0,2) rectangle (2,4);
	\draw[black,fill=pBlue] (2,4) rectangle (3,5);
	\draw[black,fill=pAlgae] (3,0) rectangle (5,2);
	\draw[black, dashed] (0,2) -- (5,2);
	\draw[black, dashed] (0,4) -- (5,4);
	\draw[black, dashed] (2,0) -- (2,5);
	\draw[black, dashed] (3,0) -- (3,5);
	\node at (2.5,4.5) {$n$};
	\end{tikzpicture}
	\caption{The structure of a permutation matrix avoiding the pattern $132$.
		The larger regions on the left and right are also $132$-avoiding,
		which explains the ``Catalan structure''.
	}\label{fig:132Avoiding}
\end{figure}
Permutations avoiding one pattern of length 3, have been studied extensively,
and it is well-known (first proved by D.~Knuth \cite{Knuth1997ArtOfProgramming,Knuth1998ArtOfProgramming})
that for all $n\geq 0$,
\[
 |\symS_{123}(n)|=|\symS_{132}(n)|=|\symS_{213}(n)|=|\symS_{231}(n)|=|\symS_{312}(n)|=|\symS_{321}(n)| = \frac{1}{n+1}\binom{2n}{n}.
\]
For patterns $\sigma$ of length 3, the inverse map and the $\revflip$ map
are bijections between some of the sets $\symS_{\sigma}(n)$, see~\cref{fig:length3Bijections}.
\begin{figure}[!ht]
	\centering
\begin{tikzpicture}[xscale=4.8,yscale=4.5,baseline=40pt,
circ/.style={circle,draw,inner sep=1pt, minimum width=4pt}
]
\node[circ] (132) at (0, 0) {$132$};
\node[circ] (213) at (-1, 0) {$213$};
\node[circ] (231) at (0, 1) {$231$};
\node[circ] (312) at (1, 1) {$312$};
\node[circ] (321) at (2, 1) {$321$};
\node[circ] (123) at (-2, 0) {$123$};
\draw[black,thick,<->] (132)--(213) node[midway,above] {$\revflip$};
\path (213) edge [loop below] node {inverse} (213);
\path (132) edge [loop below] node {inverse} (132);
\draw[black,thick,<->] (231)--(312) node[midway,above] {$\revflip$};
\draw[black,thick,<->] (231)--(312) node[midway,below] {inverse};
\draw[black,thin,dashed,<->] (231)--(132) node[midway,right] {reverse};
\path (123) edge [loop above] node {$\revflip$} (123);
\path (321) edge [loop below] node {$\revflip$} (321);
\end{tikzpicture}
	\caption{Bijections between sets of permutations avoiding a pattern of length $3$,
	see \cite{SimionSchmidt1985} for details.
	}\label{fig:length3Bijections}
\end{figure}
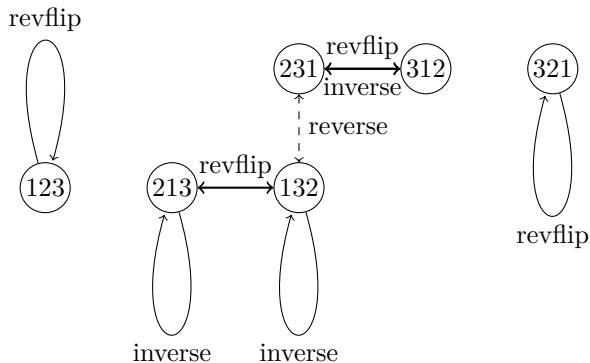

\section{Fuss--Catalan numbers and Raney numbers}

The main result in this section is \cref{prop:generalFormula},
where we give a closed-form formula for
numbers $a_k(n)$ (defined in \eqref{eq:akrecursion} below)
which generalize the Fuss--Catalan numbers. The recurrence defining the $a_k(n)$
shows up in several places in the later sections.
\medskip

The \defin{Fuss--Catalan numbers} $C_n^p$ generalize the classical Catalan numbers.
We set
\begin{equation}
\defin{C_n^p} \coloneqq \frac{1}{pn+1} \binom{pn+1}{n} = \frac{1}{(p-1)n+1} \binom{pn}{n}.
\end{equation}
Note that for $p=2$, we recover the classical Catalan numbers, see \cref{tab:fussCatalan}.
There is an extensive literature on Catalan numbers, and we refer the reader to R.~Stanley's
big collection of Catalan objects and information~\cite{StanleyCatalan}.
If we set
\[
\defin{F_p(z)} \coloneqq \sum_{n \geq 0} \frac{1}{pn+1} \binom{pn+1}{n} z^n,
\]
then Lambert~\cite{Lambert1758,Lambert1770}, (see \cite[Eq. 7.68]{GrahamKnuthPatashnik1994} for a proof),
proved that $F_p(z) = 1 + z(F_p(z))^p$.
It follows that
\begin{equation}\label{eq:lambert}
C_n^p = \sum_{\alpha} C_{\alpha_1}^p \cdot C_{\alpha_2}^p \cdot \dotsm \cdot C_{\alpha_{p}}^p, \qquad
C_0^p = 1
\end{equation}
where the sum ranges over all $\alpha \in \setN^{p}_{\geq 0}$ where the sum of the entries in $\alpha$
equals $n-1$. Note that this relation uniquely defines the sequence $\{ C_n^p \}_{n\geq 0}$.

The Fuss--Catalan numbers can be modeled as the number of
non-crossing partitions of certain polygons
using $(p+1)$-gons, see \cite[p. 78]{BrakMahony2019x}.
Another interpretation using lattice paths was mentioned in the introduction.

\begin{table}[!ht]
\begin{tabular}{p{0.5cm}rrrrrrrr}
\toprule
\; &1&2&3&4 &5&6&7&8 \\
\midrule
$C_n^1$& 1 & 1 & 1 & 1 & 1 & 1 & 1 & 1 \\
$C_n^2$& 1 & 2 & 5 & 14 & 42 & 132 & 429 & 1430 \\
$C_n^3$& 1 & 3 & 12 & 55 & 273 & 1428 & 7752 & 43263 \\
$C_n^4$& 1 & 4 & 22 & 140 & 969 & 7084 & 53820 & 420732 \\
$C_n^5$& 1 & 5 & 35 & 285 & 2530 & 23751 & 231880 & 2330445   \\
\end{tabular}
\caption{The values of $C_n^p$, with $n=1,2,\dotsc,8$ and $p=1,2,\dotsc,5$.}\label{tab:fussCatalan}
\end{table}

The \defin{Raney numbers} (introduced in \cite{Raney1960}) may be defined as
\[
 \defin{A_m(p,r)} \coloneqq \frac{r}{mp+r} \binom{mp+r}{m},
\]
and are sometimes called the \defin{two-parameter Fuss–-Catalan numbers}.
Note that $C_n^p = A_n(p,1)$, so the Raney numbers contain the Fuss--Catalan numbers as a sub-family.

\begin{proposition}\label{prop:generalFormula}
Fix $k \in \setN$ and let the sequence $\{a_k(n)\}_{n=0}^\infty$
be defined by the recursion
\begin{equation}\label{eq:akrecursion}
 a_{k}(n+1) = \sum_{\substack{0\leq i \leq n \\ k \mid i}} a_{k}(i) a_{k}(n-i),\qquad a_k(0)\coloneqq 1.
\end{equation}
The closed-form expression for $a_k(n)$ is then given by
\begin{equation}\label{eq:closedForm}
 a_{k}(km+j) = \frac{j+1}{km+j+1} \binom{ (k+1)m + j }{km+j}
\end{equation}
where we write $n=km+j$ with $0 \leq j < k$.
In particular,
\begin{equation}\label{cor:akGivesFussCat}
a_k(km+j) = A_m(k+1,j+1)  \text{ and } a_k(km) = C^{k+1}_m
\end{equation}
so the $a_{k}(n)$ form a sub-family of the Raney numbers, which includes the Fuss--Catalan numbers.
\end{proposition}
\begin{proof}
Let us first set
\[
  A_k(t)     \coloneqq \sum_{m\geq 0} a_k(km) t^m \text{ and }
  B_{k,j}(t) \coloneqq \sum_{m\geq 0} a_k(km+j) t^m \text{ for } 0 \leq j \leq k.
\]
Our first goal is to obtain functional equations
for the $B_{k,j}(t)$. We shall first prove that
\begin{equation}\label{eq:functionEquations}
 B_{k,j}(t) = (A_k(t))^{j+1} 
 \qquad \text{ and } \qquad
 A_k(t) = 1 + t \cdot (A_k(t))^{k+1}.
\end{equation}
The recursion 
in \eqref{eq:akrecursion} can be rewritten
as
\[
  \sum_{m\geq 0}  a_k(km+j)t^{m} = \sum_{m\geq 0} t^{m} \sum_{i} a_k(ki) \cdot a_k(km+j-1-ki),
\]
for a fixed $j$ such that $1 \leq j \leq k$, by first substituting 
$n \mapsto km+j-1$ and next multiply by  $t^{m}$ in both sides then sum over $m$.
We then have that
\begin{align*}
B_{k,j}(t) = \sum_{m\geq 0} \sum_{i} a_k(ki) t^i \cdot  a_k(km-ki+j-1) t^{m-i}
= A_k(t) \cdot B_{k,j-1}(t).
\end{align*}
A simple inductive argument over $j$, concludes the proof
and the first set of relations in \eqref{eq:functionEquations} is established.
Now, on one hand, we have that
$ B_{k,k}(t) = (A_k(t))^{k+1}$.
On the other hand, the definition of the $B_{k,j}$
tells us that $1 + t B_{k,k}(t) = A_k(t)$, so 
\[
 A_k(t) = 1 + t (A_k(t))^{k+1},
\]
which is the second relation in \eqref{eq:functionEquations}.
\bigskip

It now remains to obtain closed formulas for the coefficients of $B_{k,j}(t)$.
Starting from $ A_k = 1 + t A_k^{k+1}$, we
raise both sides to the $(j+1)^\thsup$ power, and get
\begin{align*}
 A_k^{j+1} &= \left( 1+ t A_k^{k+1} \right)^{j+1} \\
 A_k^{j+1} &= \left( 1+ t (A_k^{j+1})^{\frac{k+1}{j+1}} \right)^{j+1} \\
 B_{k,j} &= \left( 1+ t (B_{k,j})^{\frac{k+1}{j+1}} \right)^{j+1}.
\end{align*}
After multiplying by $t^{\frac{j+1}{k+1}}$, we get
\[
 t^{\frac{j+1}{k+1}} B_{k,j} =
 t^{\frac{j+1}{k+1}} \left( 1+ (t^{\frac{j+1}{k+1}} B_{k,j})^{\frac{k+1}{j+1}} \right)^{j+1}.
\]
We now do the substitution $t \mapsto s^{\frac{k+1}{j+1}}$, so that
\[
 s B_{k,j}(s^{\frac{k+1}{j+1}}) = s \left( 1+ (s B_{k,j}(s^{\frac{k+1}{j+1}}))^{\frac{k+1}{j+1}} \right)^{j+1}.
\]
Finally, we set $F(s) \coloneqq s B_{k,j}(s^{\frac{k+1}{j+1}})$,
and we end up with the functional equation
\[
 F(s) = s \left( 1+ (F(s))^{\frac{k+1}{j+1}} \right)^{j+1}.
\]
We now use Lagrange inversion (see \cite[Sec. 5.1]{Wilf1994}),
where we use a slight extension allowing for rational exponents.
A proof is included further down in Appendix~\ref{apx:lagrange}.
In our case, we have that
\begin{equation}\label{eq:lagrangeInv}
  F(s) \Big\vert_{s^r} = \frac{1}{r} \phi(x)^r \Big\vert_{x^{r-1}}
\text{ where } \phi(x) = \left(1+x^{\frac{k+1}{j+1}} \right)^{j+1}.
\end{equation}
Now,
\[
B_{k,j}(s) \Big \vert_{s^n} =  s^{-\frac{j+1}{k+1} } F\left(s^{\frac{j+1}{k+1}}\right) \Big \vert_{s^{n}}
=  F(s) \Big \vert_{s^{n(k+1)/(j+1)+1}}.
\]
By using the relation in \eqref{eq:lagrangeInv},
we then have that
\begin{align*}
 B_{k,j}(s) \Big \vert_{s^m} &=
 \frac{1}{m \frac{k+1}{j+1} +1 }
 \left(1+x^{\frac{k+1}{j+1}} \right)^{m(k+1)+(j+1)} \Big \vert_{x^{m(k+1)/(j+1)}} \\
 &=
 \frac{j+1}{m(k+1)+ j+1}
 \left(1+x \right)^{n(k+1)+(j+1)} \Big \vert_{x^{m}} \\
 &=
 \frac{j+1}{(k+1)m+ j+1}
 \binom{ (k+1)m+(j+1) }{ m }.
\end{align*}
It is now straightforward to verify that this expression
agrees with the one given in \eqref{eq:closedForm} and finally to verify \eqref{cor:akGivesFussCat}.
\end{proof}

\begin{proposition}\label{prop:anotherrecur}
For a fixed $r$ in $\{0,1,\dotsc,k-1\}$ and $m\geq 1$, we have 
\begin{equation}\label{eq:anotherrecur}
	a_k(km)=\sum_{j=0}^{m-1} a_k(kj+r)a_k(k(m-j)-r-1),
\end{equation}
where $a_k(n)$ is the sequence in \cref{prop:generalFormula}.
\end{proposition}
\begin{proof}
	Recall the notation in \cref{prop:generalFormula}, 
	where
	\[
	A_k(t)     \coloneqq \sum_{m\geq 0} a_k(km) t^m \text{ and }
	B_{k,r}(t) \coloneqq \sum_{m\geq 0} a_k(km+r) t^m.
	\]
	The second equation in \eqref{eq:functionEquations} implies that
	\[
	A_k(t)=1+tB_{k,r}(t)B_{k,k-1-r}(t).
	\]
	 Thus,
	\begin{align*}
	\sum_{m\geq 0}a_k(km)t^m&=1+t\left( \sum_{m\geq 0}a_k(km+r)t^m\right)\left( \sum_{m\geq 0}a_k(km+k-1-r)t^m\right)\\
	&=1+\sum_{m\geq 1}t^m\sum_{j=0}^{m-1}a_k(kj+r)a_k(k(m-j)-1-r).
	\end{align*}
	Comparing coefficients of $t^m$, we arrive at \eqref{eq:anotherrecur}.
\end{proof}

\section{Mod-\texorpdfstring{$k$}{k}-alternating permutations and pattern avoidance}\label{sec:modKPattern}

\subsection{Mod-\texorpdfstring{$k$}{k}-alternating permutations}

Let $\PAP(n)$ denote the set of \defin{parity alternating permutations} (\defin{PAP}),
starting with an odd entry and let $\pap(n)$ be the cardinality of $\PAP(n)$.
Moreover, we let $\PAP_\sigma(n)$ be the set of PAPs avoiding the permutation pattern $\sigma$,
and $\PAP_{\sigma,\tau}(n)$ be set of PAPs avoiding both $\sigma$ and $\tau$.

Similarly, we let $\PAP^*(n)$ denote the set of parity-alternating permutations starting with
an \emph{even} entry, and we define $\PAP^*_\sigma(n)$ and $\PAP^*_{\sigma,\tau}(n)$ analogously.
The following lemma is straightforward to verify.
\begin{lemma}\label{lem:revBij}
For $m \geq 0$,
\begin{align*}
&\rev: \PAP_{132}(2m)  \to \PAP^*_{231}(2m) \\
&\rev: \PAP_{231}(2m)  \to \PAP^*_{132}(2m) \\
&\rev: \PAP_{132}(2m+1)  \to \PAP_{231}(2m+1) \\
&\rev: \PAP_{231}(2m+1)  \to \PAP_{132}(2m+1)
\end{align*}
are bijections. 
\end{lemma}
For instance, $34561278$ is mapped to $87216543$ and $5234167$ 
is mapped to $7614325$ under the reverse map, 
where $34561278$ and $5234167$ are $132$-avoiding 
and $87216543$ and $7614325$ are $231$-avoiding.

\begin{definition}
	Given $k\geq 2$, a \defin{mod-$k$-alternating permutation} of size $n$
	is a permutation $\pi \in \symS_n$ such that
	\[
	\pi(i)\equiv i \mod k \qquad \text{ for all $i=1,2,\dotsc,n$}.
	\]
Let $\defin{\RRP(n,k)}$ be the set of such mod-$k$-alternating permutations and we set $\defin{\RRP_\sigma(n,k)}\coloneqq\RRP(n,k)\cap\symS_\sigma(n)$.
Note that $\RRP(n,k)$ is a subgroup of $\symS_n$, so that it is closed under
taking inverses.

For $k=2$ we recover the classical parity-alternating permutations starting with an odd integer.
Note that
\begin{align*}
\defin{\rrp(n,k)}\coloneqq|\RRP(n,k)| = \left(\Big\lceil\frac{n}{k}\Big\rceil!\right)^j\left(\Big\lfloor\frac{n}{k}\Big\rfloor!\right)^{k-j},
\end{align*}
where $j$ is the remainder when $n$ is divided by $k$, while we let $\defin{\rrp_\sigma(n,k)}\coloneqq|\RRP_\sigma(n,k)|$. 
\end{definition}

\begin{example}
For example, $\RRP(7,3)$ has $(3!)^1 \cdot (2!)^2 = 24$ elements:
\begin{equation*}
\begin{matrix}
 [1,2,3,4,5,6,7] & [1,2,6,4,5,3,7] & [1,5,3,4,2,6,7] & [1,5,6,4,2,3,7] & [1,2,3,7,5,6,4]\\
  [1,2,6,7,5,3,4] &  [1,5,3,7,2,6,4] & [1,5,6,7,2,3,4] & [4,2,3,1,5,6,7] & [4,2,6,1,5,3,7]\\ 
  [4,5,3,1,2,6,7] & [4,5,6,1,2,3,7] & [4,2,3,7,5,6,1] & [4,2,6,7,5,3,1] & [4,5,3,7,2,6,1]\\ 
  [4,5,6,7,2,3,1] & [7,2,3,1,5,6,4] & [7,2,6,1,5,3,4] & [7,5,3,1,2,6,4] & [7,5,6,1,2,3,4]\\
   [7,2,3,4,5,6,1]& [7,2,6,4,5,3,1] & [7,5,3,4,2,6,1] & [7,5,6,4,2,3,1] & \\
\end{matrix}
\end{equation*}
\end{example}

We shall also introduce some notation for mod-$k$-alternating permutations
starting with the remainder $r$ mod $k$.
\begin{definition}
Let \defin{$\RRP(n,k,r)$} denote the set of permutations of length $n$
which satisfy
\begin{equation}\label{eq:modKwithRemainder}
  \pi(i) \equiv (r+i-1)\mod k \qquad\text{ for all $1 \leq i \leq n$}.
\end{equation}
Moreover, we set $\defin{\RRP_\sigma(n,k,r)} \coloneqq \RRP(n,k,r) \cap \symS_\sigma(n)$, $\defin{\rrp(n,k,r)}\coloneqq |\RRP(n,k,r)|$, and $\defin{\rrp_\sigma(n,k,r)}\coloneqq |\RRP_\sigma(n,k,r)|$.
\end{definition}

\begin{lemma}\label{lem:nonEmptyCases}
We have that $\rrp(n,k,r)=0$ unless $r=1$, or $n=mk$ for some integer $m$.
\end{lemma}
\begin{proof}
Suppose $2 \leq r \leq k$ and $\pi \in \RRP(n,k,r)$.
By \eqref{eq:modKwithRemainder}, the number of \emph{positions} in $\pi$,
where the entry has remainder $r$, always exceed the number of entries with remainder $r-1$,
unless $n$ is a multiple of $k$. This is due to the fact that the first position have remainder $r$,
so that the remainders must form the pattern
\[
  [r,r+1,r+2,\dotsc,0,1,\dotsc,r-1,r,\dotsc].
\]	
Now, the number of \emph{entries} with index equal to $r$ mod $k$,
is never larger than the number of entries with index equal to $r-1$ mod $k$,
since the entries are from $[n]$.
Since these two counts must agree for $\pi$, the statement follows.
\end{proof}

\begin{lemma}
The cardinality of $\RRP(nk,k,r)$ is $(n!)^k$.
\end{lemma}
\begin{proof}
This is straightforward.
\end{proof}

\begin{lemma}\label{lem:modkAltRevFlip}
We have that $\rrp_{\sigma}(n,k) = \rrp_{\revflip(\sigma)}(n,k)$.
Moreover,
\begin{equation}\label{eq:revflipRelation}
\rrp_{\sigma}(n,k,r) = \rrp_{\revflip(\sigma)}(n,k,2-r).
\end{equation}
In particular, $\rrp_{132}(n,k)$ and $\rrp_{213}(n,k)$
have the same cardinality.
\end{lemma}
\begin{proof}
Note that it is enough to prove \eqref{eq:revflipRelation},
and that it suffices to show 
that $\revflip : \RRP_{\sigma}(n,k,r) \to \RRP_{\revflip(\sigma)}(n,k, 2-r )$
is a bijection.

For $\pi \in \RRP_{\sigma}(n,k,r)$,
we have that the last element in $\pi$ has remainder $r-1+n$ mod $k$.
Hence, first element of $\revflip(\pi)$ has remainder $n+1-(r-1+n) = 2-r$ mod $k$.
The remaining properties are now straightforward to verify,
that is, that $\revflip(\pi) \in \RRP_{\revflip(\sigma)}(n,k, 2-r )$.
\end{proof}

\begin{lemma}\label{lem:modkInverse}
We have that $\rrp_{\sigma}(n,k)= \rrp_{\sigma^{-1}}(n,k)$.
\end{lemma}
\begin{proof}
We have that $\pi \in \symS_\sigma(n) \iff \pi^{-1} \in \symS_{\sigma^{-1}}(n)$
(this is classical, see \cite[Lem.1]{SimionSchmidt1985})
so $\pi \in \RRP_{\sigma}(n,k)$ holds iff $\pi^{-1} \in \RRP_{\sigma^{-1}}(n,k)$.
\end{proof}

\begin{proposition}\label{prop:213Recursion}
Let $a_k(n) = \rrp_{213}(n,k) = \rrp_{132}(n,k)$.
Then we have the recursion
\begin{equation}\label{eq:213modKAlternatingRecursion}
 a_{k}(n+1) = \sum_{\substack{0\leq j \leq n \\ k \mid j}} a_{k}(j) a_{k}(n-j),\qquad a_k(0)\coloneqq 1.
\end{equation}
\end{proposition}
\begin{proof}
Consider $ \pi \in \RRP_{213}(n+1,k)$. It must be of the form
\[
  \pi = [\alpha_1,\dotsc,\alpha_{j},1,\beta_1,\dotsc,\beta_{n-j}]
\]
where we must have $\alpha_r > \beta_s$ for all relevant $r,s$,
see \cref{fig:213RRPFig}.

\begin{figure}[!ht]
	\centering
	\begin{tikzpicture}[scale=1.1]
	\draw[black,fill=pPink] (5,0) rectangle (6,1);
	\draw[black,fill=pBlue] (6,1) rectangle (9,4);
	\draw[black,fill=pOrange] (0,4) rectangle (5,9);
	
	\draw[black, dashed] (6,0) -- (6,9);
	\draw[black, dashed] (5,0) -- (5,9);
	\draw[black, dashed] (0,4) -- (9,4);
	\draw[black, dashed] (0,1) -- (9,1);

	\node at (5.5,0.5) {$1$};
	\node at (7.5,2.5) {$\beta$};
	\node at (2.5,6.5) {$\alpha$};
	\end{tikzpicture}
\caption{The structure of a $213$-avoiding permutation.}\label{fig:213RRPFig}
\end{figure}
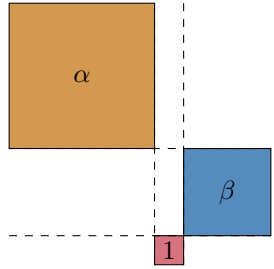
Since $1$ can only appear at some position $km+1$ for some $m \in \setN$,
we must have that $k\mid j$. Hence, setting
\begin{align*}
  \alpha' & \coloneqq [\alpha_1 - (n-j+1),\alpha_2 - (n-j+1), \dotsc,\alpha_j - (n-j+1)] \\
  \beta' & \coloneqq [\beta_1-1, \beta_2-1, \dotsc,\beta_{n-j}-1],
\end{align*}
we have that $\alpha' \in \RRP_{213}(j,k)$, $\beta' \in \RRP_{213}(n-j,k)$.
Thus we have a bijective proof of the recursion in \eqref{eq:213modKAlternatingRecursion}.
\end{proof}

We can now give two new combinatorial interpretations of the
Fuss--Catalan numbers.
\begin{corollary}\label{cor:mainFussCatalan}
Let $\sigma \in \{132,213\}$ and $n,k \geq 1$.
Let $m$ and $j$ be defined via $n=km+j$ with $0 \leq j < k$.
Then
\[
\rrp_\sigma(n,k) =  \frac{j+1}{km+j+1} \binom{ (k+1)m + j }{km+j}
\]
and in particular
\[
\rrp_\sigma(km,k) = \frac{1}{(k+1)m+1} \binom{(k+1)m+1}{m}
	= \frac{1}{km+1} \binom{(k+1)m}{m}.
\]
\end{corollary}
\begin{proof}
 This follows from \cref{prop:generalFormula} and \cref{cor:akGivesFussCat}.
\end{proof}

\begin{proposition}\label{prop:onlyIdentityStronger}
For $k\geq 3$ and $n \geq 1$ we have that
\begin{equation}\label{eq:onlyIdentityStronger}
 \rrp_{312}(n,k,r) =
\begin{cases}
1 & \text{ if $r=1$}\\
2^{n/k-1} & \text{ if $r=2$ and $k \mid n$}\\
0 & \text{ otherwise}.
\end{cases}
\end{equation}
Note that by \cref{lem:modkAltRevFlip}, a corresponding statement holds for the pattern $231$.
\end{proposition}
\begin{proof}
We shall make the claim slightly stronger,
namely that  $\RRP_{312}(n,k,1)$ consists of only the identity permutation.
We proceed by induction over $n$ and assume
that the statement in \eqref{eq:onlyIdentityStronger} is true
for permutations of length less than $n$; the base cases $n=1$ and $n=2$ are straightforward to verify.
\emph{Note that since $k \geq 3$, $r=0$ and $r=2$ are different cases.}

Suppose that $\pi \in \RRP_{312}(n,k,r)$ be of the form
\begin{equation}\label{eq:312form}
	\pi = [\alpha,1,\beta],
\end{equation}
for some $n\geq 3$.

\textbf{Case $r=1$.} If $\alpha$ is the empty sequence, then $[\beta-1] \in \RRP_{312}(n-1,k,1)$ and this must be the identity permutation by induction.
Hence, we can assume that the length, say $s$, of $\alpha$ is greater than $0$ and we must therefore have that
$n \geq k+1$.

Since $\pi$ avoids $312$, $\max(\alpha)<\min(\beta)$. In particular, $\alpha$ is constructed from $\{2,3,\dotsc, s+1\}$. So $[\alpha,1]\in \RRP_{312}(s+1,k,1)$. 

The case for $\beta$ being non-empty is impossible due to the induction hypothesis,
so $\pi$ must be of the form $\pi = [\alpha,1]$.
But then, $[\alpha-1]$ is an element in $\RRP_{312}(s,k,0)$,
but this set is empty (by induction hypothesis).
Hence, the case $r=1$ is done.

\textbf{Case $r \neq 1$.}
By \cref{lem:nonEmptyCases}, we can assume $k \mid n$ and $\alpha-1\in \RRP_{312}(s,k,r-1)$. But this set is empty unless $r=2$ or $r=3$ and $k \mid s$. The later case is impossible since the form in (\ref{eq:312form}) tells us $r+s\equiv 1\mod k$, which implies that $k\nmid s$. For the first case, $r=2$, there is only one possibility for
$\alpha$, namely $[2,3,\dotsc,s+1]$ by induction. 

And we know that $[\beta-s-1] \in \RRP_{312}(n{-}s{-}1,k,2)$. Moreover, we also need that $k \mid (n{-}s{-}1)$ and there are (by induction) $2^{(n{-}s{-}1)/k-1}$
options for $\beta$.
Accounting for all possible values of $n{-}s{-}1$ in $\{0,k,2k,\dotsc,n-1\}$,
we get $1+1+2+4+\dotsb+2^{(n{-}1)/k-1} = 2^{n/k-1}$
ways to obtain $\pi$. 

\end{proof}

The following proposition is analogous to \cref{prop:onlyIdentityStronger}.
\begin{proposition}\label{prop:213kgeq3}
Let $k \geq 3$.
Then, with $a_k(n)$ defined as in \cref{prop:generalFormula}, we have
\begin{equation}\label{eq:213kgeq3}
\rrp_{213}(n,k,r) =
\begin{cases}
a_k(n) &\text{ if $r=1$ or $k \mid n$} \\
0 &\text{ otherwise.}
\end{cases}
\end{equation}
By \cref{lem:modkAltRevFlip}, a similar statement holds for the pattern $132$.
\end{proposition}
\begin{proof}
We have already proved the case $r=1$ in \cref{prop:213Recursion},
and the ``otherwise'' case follows from \cref{lem:nonEmptyCases}.
Hence, only the case $k \mid n$, $r \neq 1$ needs to be proved.
Note that it suffices to show that $\rrp_{213}(km,k,r) = a_k(km)$.
\medskip

It is straightforward to verify that \eqref{eq:213kgeq3} holds whenever $n \leq 2$.
Now suppose $\pi \in \RRP_{213}(n,k,r)$ for some $r \in \{2,3,\dotsc,k\}$ and $n \geq 3$, with $n=km$.
Then $\pi$ is of the form
\[
  [\alpha_1,1, \alpha_{2}]
\]
where $\min(\alpha_1)>\max(\alpha_2)$, as seen in \cref{fig:213RRPFig}.
Let $\alpha_1$ have length $s$, so that $\alpha_2$ has length $n-1-s$.
Note that we must have $k \mid s+r-1$, in order for the first entry in $\alpha_1$
to have remainder $r$ mod $k$.
We then have that 
\[
(\alpha_2-1) \in \RRP_{213}(n-1-s,k,1) \text{ and }
(\alpha_1-(n-s)) \in \RRP_{213}(s,k,1).
\]
Conversely, given permutations $\tau_1 \in \RRP_{213}(s,k,1)$,
and $\tau_2\in \RRP_{213}(n-1-s,k,1)$ for some $s$ satisfying $k \mid s+r-1$,
we can construct a unique $\pi \in \RRP_{213}(n,k,r)$, by reversing
the above process. 
Hence, when $k\mid n$, we have that
\begin{equation*}
 \rrp_{213}(n,k,r) = \sum_{\substack{s=1 \\ k \mid s+r-1}}^{n-1} \rrp_{213}(s,k,1) \cdot \rrp_{213}(n-s-1,k,1).
\end{equation*}
Now, setting $n=km$, $s=kj+r'$ (note that $s$ can be expressed in this way for the fixed $r' = k-r+1$)
and using that $a_k(i) = \rrp_{213}(i,k,1)$,
we get
\begin{equation*}
 \rrp_{213}(km,k,r) = \sum_{j=0}^{m-1} a_k(kj+r') \cdot a_k(km-(kj+r')-1).
\end{equation*}
By \eqref{eq:anotherrecur} in \cref{prop:anotherrecur}, we may now conclude that $\rrp_{213}(km,k,r) = a_k(km)$.

\end{proof}

\subsection{PAPs avoiding one of 132, 213, 231, and 312}

The case $k=2$ is rather different from $k \geq 3$,
as indicated by \cref{prop:onlyIdentityStronger,prop:213kgeq3}.
In this subsection we treat the $k=2$ case for
the patterns in $\{132,213,231,312\}$.

\begin{lemma}\label{lem:132Structure}
Let $\pi \in \PAP_{132}(n)$, with $n \geq 2$.
Then either $n$ appears after $n-1$ in $\pi$,
or $n$ is odd and appears at position $1$.
\end{lemma}
\begin{proof}
The case $n=2$ is easy, so suppose $n \geq 3$ and $\pi(1)\neq n$.
Then the numbers to the left of $n$ must be greater than the
numbers from the right of $n$ in $\pi$
(since we must avoid the pattern $132$, see \cref{fig:132Avoiding}).
In particular, $n-1$ must appear before $n$, and we are done.
\end{proof}

\begin{proposition}\label{prop:132-231}
For $m\geq 0$, set
	\begin{align*}
	a_m \coloneqq \pap_{132}(2m), \quad b_m \coloneqq \pap_{132}(2m+1), \\
	a'_m \coloneqq \pap_{231}(2m), \quad b'_m \coloneqq \pap_{231}(2m+1).
	\end{align*}
We then have $a_m = a'_m$, $b_m = b'_m$ for all $m$.
\end{proposition}
\begin{proof}
We already know that $b_m=b'_m$, since $\rev: \PAP_{132}(2m+1)\to \PAP_{231}(2m+1)$
is a bijection.

\medskip
	\textbf{Claim:}
	We have
	\begin{equation}
	a_{m} = \sum_{\substack{i+j+k=m-1 \\ i,j,k \geq 0}} a_i a'_j a_k, \text{ and }
		a'_{m} = \sum_{\substack{i+j+k=m-1 \\ i,j,k \geq 0}} a'_i a_j a'_k.
	\label{eq:132RecursionMod}
	\end{equation}
	First set $n\coloneqq 2m$ and suppose $\pi \in \pap_{132}(2m)$.
	By \cref{lem:132Structure}, we know that it must be of the form
	\[
	[\underbrace{\alpha_1, \alpha_2, \dotsc, \alpha_{2i}}_\alpha, \; n-1, \;
	\underbrace{\beta_{1},\dotsc, \beta_{2j}}_\beta, \; n, \;
	\underbrace{\gamma_{1},\dotsc,\gamma_{2k}}_\gamma].
	\]
	Moreover, since the permutation is $132$-avoiding (see \cref{fig:132Avoiding}) it must be of the form
	
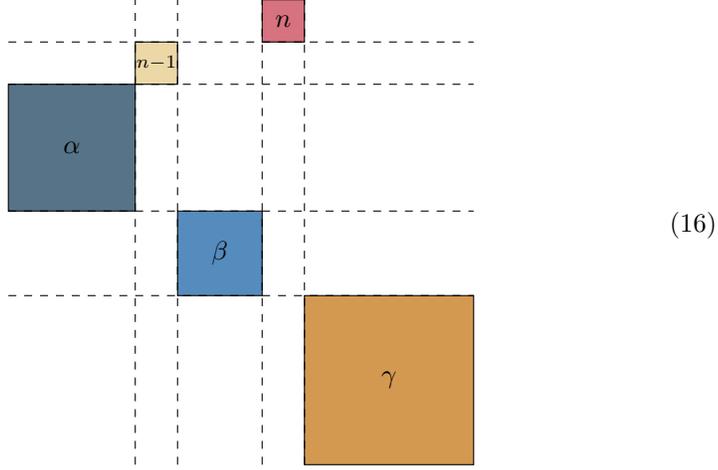
\begin{equation}\label{eq:132RecursionFig}
	\begin{tikzpicture}[scale=1.6]
	\draw[black,fill=pOrange] (7,0) rectangle (11,4);
	\draw[black,fill=pBlue] (4,4) rectangle (6,6);
	\draw[black,fill=pAlgae] (0,6) rectangle (3,9);
	\draw[black,fill=pSkin] (3,9) rectangle (4,10);
	\draw[black,fill=pPink] (6,10) rectangle (7,11);
	\draw[black, dashed] (3,0) -- (3,11);
	\draw[black, dashed] (4,0) -- (4,11);
	\draw[black, dashed] (6,0) -- (6,11);
	\draw[black, dashed] (7,0) -- (7,11);
	\draw[black, dashed] (0,4) -- (11,4);
	\draw[black, dashed] (0,6) -- (11,6);
	\draw[black, dashed] (0,9) -- (11,9);
	\draw[black, dashed] (0,10) -- (11,10);
	\node at (6.5,10.5) {$n$};
	\node at (3.5,9.5) {$\scriptstyle{n-1}$};
	\node at (9,2) {$\gamma$};
	\node at (5,5) {$\beta$};
	\node at (1.5,7.5) {$\alpha$};
	\end{tikzpicture}
	\end{equation}
Now, the permutation matrices $\alpha$ and $\gamma$ can be interpreted
as smaller permutations in some $\PAP_{132}(2i)$ and $\PAP_{132}(2k)$, respectively.
Moreover, the permutation $\beta$ can be seen as an element in $\PAP^*_{132}(2j)$.
By the reversal bijection, we have that $\beta$ correspond to
some element in $\PAP_{231}(2j)$, whose cardinality is $a'_j$.
Hence, we can conclude the first formula in \eqref{eq:132RecursionMod}.

Using a similar argument (just reverse everything),
we can show the second recursive identity.
An inductive argument over $m$ now allows us to
conclude that $a_m = a'_m$ for all $m \geq 0$.
\end{proof}

\begin{corollary}
	For any $\sigma \in \{132,213,231,312\}$, we have that
	\begin{equation}
	\pap_{\sigma}(2m) = \frac{1}{2m+1}\binom{3m}{m}\quad\text{and}\quad
	\pap_{\sigma}(2m+1) = \frac{1}{2m+1}\binom{3m+1}{m+1}.
	\end{equation}
\end{corollary}
\begin{proof}
The formulas follow from \cref{cor:mainFussCatalan} for $\sigma=132$. By \cref{prop:132-231} the statement holds
for the pattern $231$. And
by applying $\revflip$, we obtain the same counts for the other two patterns.
\end{proof}

\section{Mod-k-alternating subexcedant functions}\label{sec:subex}

Subexcedant functions, and their close connection with permutations
was initially studied in \cite{MantaciRakotondrajao2001}. 
The topic of subexedant functions is rather new, but we refer to
\cite{AlexanderssonGetachew2021x} and \cite{BeyeneMantaci2021x}
as examples of applications of these objects.

\begin{definition}\label{subex}
A \defin{subexcedant function}  $f$ on $[n]$ is a map $f : [n] \longrightarrow [n]$
such that
\begin{equation*}
1\leq f(i)\leq i \text{ for all } 1\leq i\leq n.
\end{equation*}
We denote the set of all
subexcedant functions on $[n]$ by $\defin{\SEF_n}$.
\end{definition}
One can easily see that $\mathcal{F}_n$ has cardinality $n!$. 
The bijection $\defin{\sefToPerm}: \SEF_n \longrightarrow \symS_n$, 
which is defined as a product 
\begin{equation}\label{eq:sefToPerm}
  \defin{\sefToPerm(f)} \coloneqq (n\,\,f(n))  \dotsm (2\,\,f(2)) (1\,\,f(1)),  
\end{equation}
of cycles of maximum length 2, which is the one described in \cite{MantaciRakotondrajao2001}.

The inverse mapping is defined as follows:
\begin{equation}\label{eq:permToSefRecursion}
 \sefToPerm^{-1}(\sigma)_j \coloneqq
\begin{cases}
 \sigma(n) \text{ if $j=n$,} \\
 \sefToPerm^{-1}( \; \sigma \circ (n\,\,\sigma(n)) \;  )_j &\text{ otherwise},
\end{cases}
\end{equation}
for $\sigma \in \symS_n$ and $j \in [n]$.

\begin{proposition}\label{prop:sefRRPs}
Suppose $\pi \in \symS_n$ and $f_\pi$ is the corresponding element in $\SEF_n$.
Then $\pi \in \RRP(n,k)$ if and only if $f_\pi(i) \equiv i \mod k$.
\end{proposition}
\begin{proof}
First suppose that $f_\pi(i) \equiv i \mod k$ for some $\pi\in\symS_n$.
In the product \eqref{eq:sefToPerm}, we see that $f_\pi$ has $k$ disjoint cycle products;
the first product contains integers that are congruent to $1\mod k$, the second product contains
integers that are congruent to $2\mod k$ and so on, with the 
last product containing integers that are congruent to $k\mod k$.
Hence, $\pi \in \RRP(n,k)$.

On the other hand, suppose now that $\pi \in \RRP(n,k)$, and consider $f_\pi = \sefToPerm^{-1}(\pi)$.
In the recursion in \eqref{eq:permToSefRecursion},
we always interchange integers at positions $i$ and $j$ where $j\equiv i\mod k$,
so (via an inductive argument) we must have that $f_\pi(i) \equiv i \mod k$ for all $i$.
\end{proof}

\begin{example}
	The subexcedant function $12341634$ corresponds to the permutation $[5,2,7,8,1,6,3,4]$ in $\RRP(8,4)$.
\end{example}
\begin{definition}
	A word $w = (w_1,\dotsc,w_n) \in \setN^n$ is a \defin{Catalan word} if
	\[
	w_1 =1 \text{ and } 1\leq w_i \leq w_{i-1}+1 \text{ whenever $2 \leq i\leq n$}.
	\]
\end{definition}
There are $\frac{1}{n+1}\binom{2n}{n}$ Catalan words of length $n$.

\begin{definition}
	A permutation $\pi$ in $\PAP(n)$ is called a \defin{Catalan PAP} if $f_\pi$ is a Catalan word.
\end{definition}
Let $\defin{\PAP_C(n)}$ be the set of Catalan PAPs, and let $\defin{\pap_C(n)}$ be its cardinality.
We define $\defin{\RRP_C(n,k)}$ in the same manner.

\begin{definition}
	A sequence $(a_1,\dotsc,a_n)$ is an \defin{area sequence} if
	\[
	a_1 =0 \text{ and } 0\leq a_i \leq a_{i-1}+1 \text{ whenever $2 \leq i\leq n$}.
	\]
\end{definition}

Area sequences are in bijection with Catalan words: $(a_1,\dotsc,a_n)$
is an area sequence if and only if $(a_1+1,a_2+1,\dotsc,a_n+1)$
is a Catalan word. The notion of area sequences plays an 
important role in \emph{diagonal harmonics}, see \cite{qtCatalanBook}.

There is a well-known bijection between Dyck paths and area sequences.
We illustrate this bijection in the following example.
\begin{example}\label{ex:dyckPath}
The Dyck path $P=\texttt{nnennneneeneee}$ has the area sequence $\avec = (0,1,1,2,3,3,2)$.
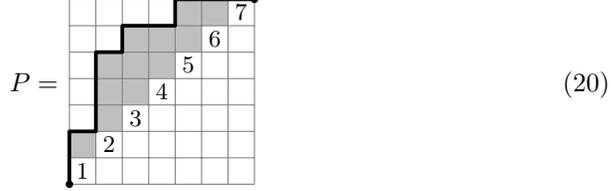
\begin{equation}
P=
\begin{tikzpicture}
\begin{scope}
\fill[lightgray](0,1)--(0,2)--(1,2)--(1,5)--(2,5)--(2,6)--(4,6)--(4,7)--(6,7)--(6,6)--(5,6)--(5,5)--(4,5)--(4,4)--(3,4)--(3,3)--(2,3)--(2,2)--(1,2)--(1,1)--cycle;
\draw[step=1em,gray](0,0) grid (7,7);
\draw[thickLine](0,0)--(0,2)--(1,2)--(1,5)--(2,5)--(2,6)--(4,6)--(4,7)--(7,7);

\draw
(0,0)node[mydot]{}
(7,7)node[mydot]{};
\draw
(1,1)node[entries]{$1$}
(2,2)node[entries]{$2$}
(3,3)node[entries]{$3$}
(4,4)node[entries]{$4$}
(5,5)node[entries]{$5$}
(6,6)node[entries]{$6$}
(7,7)node[entries]{$7$};
\end{scope}
\end{tikzpicture}
\end{equation}
The number of shaded boxes in row $i$ (from the bottom) is given by $a_i$,
and this provides the bijection between area sequences and Dyck paths.
\end{example}

We denote the set of all Dyck paths of size $n$ such that the number of
boxes in each row above the path is a multiple of $k$ by $\defin{L^k_{n}}$,
see \cref{ex:specialDyckPaths}.
\begin{corollary}
The set $\RRP_C(n,k)$ has the same cardinality as $L^k_{n}$.
\end{corollary}
\begin{proof}
Let $\pi\in \RRP_C(n,k)$. By \cref{prop:sefRRPs}, the corresponding subexcedant function
$f_\pi=w_1w_2\dotsm w_n$ is a Catalan word with 
the property $i-w_i=\alpha_i k$, for all $i\in [n]$ where  $\alpha_i\in \setN$. 
Hence, $\avec = (w_1-1,w_2-1,\dotsc,w_n-1)$ is an 
area sequence for some Dyck path.
Now, the number of boxes in the $i^\thsup$ row above the Dyck path is given by
 \begin{equation}
\label{equ:areaseq}
  (i-1) - (w_i-1)=\alpha_i k,
\end{equation}
since the $i^\thsup$ row has $i-1$ boxes above the diagonal in the $n\times n$ grid.
\end{proof}

\begin{example}\label{ex:specialDyckPaths}
Here are the 9 Dyck paths corresponding to the entries in $\RRP_C(7,3)$.
\[
\begin{tikzpicture}[scale=0.9]
\fill[lightgray](0,0)--(0,1)--(0,2)--(0,3)--(1,3)--(2,3)--(3,3)--(3,4)--(3,5)--(3,6)--(4,6)--(5,6)--(6,6)--(6,7)--(7,7)--(6,7)--(6,6)--(5,6)--(5,5)--(4,5)--(4,4)--(3,4)--(3,3)--(2,3)--(2,2)--(1,2)--(1,1)--(0,1);
\draw[step=1em,gray](0,0) grid (7,7);
\draw[thickLine](0,0)--(0,1)--(0,2)--(0,3)--(1,3)--(2,3)--(3,3)--(3,4)--(3,5)--(3,6)--(4,6)--(5,6)--(6,6)--(6,7)--(7,7);
\draw (0,0)node[mydot]{}
(7,7)node[mydot]{};
\draw
(1,1) node[entries]{$1$}
(2,2) node[entries]{$2$}
(3,3) node[entries]{$3$}
(4,4) node[entries]{$4$}
(5,5) node[entries]{$5$}
(6,6) node[entries]{$6$}
(7,7) node[entries]{$7$}
;
\end{tikzpicture}
\begin{tikzpicture}[scale=0.9]
\fill[lightgray](0,0)--(0,1)--(0,2)--(0,3)--(0,4)--(1,4)--(2,4)--(3,4)--(3,5)--(3,6)--(4,6)--(5,6)--(6,6)--(6,7)--(7,7)--(6,7)--(6,6)--(5,6)--(5,5)--(4,5)--(4,4)--(3,4)--(3,3)--(2,3)--(2,2)--(1,2)--(1,1)--(0,1);
\draw[step=1em,gray](0,0) grid (7,7);
\draw[thickLine](0,0)--(0,1)--(0,2)--(0,3)--(0,4)--(1,4)--(2,4)--(3,4)--(3,5)--(3,6)--(4,6)--(5,6)--(6,6)--(6,7)--(7,7);
\draw (0,0)node[mydot]{}
(7,7)node[mydot]{};
\draw
(1,1) node[entries]{$1$}
(2,2) node[entries]{$2$}
(3,3) node[entries]{$3$}
(4,4) node[entries]{$4$}
(5,5) node[entries]{$5$}
(6,6) node[entries]{$6$}
(7,7) node[entries]{$7$}
;
\end{tikzpicture}
\begin{tikzpicture}[scale=0.9]
\fill[lightgray](0,0)--(0,1)--(0,2)--(0,3)--(0,4)--(0,5)--(1,5)--(2,5)--(3,5)--(3,6)--(4,6)--(5,6)--(6,6)--(6,7)--(7,7)--(6,7)--(6,6)--(5,6)--(5,5)--(4,5)--(4,4)--(3,4)--(3,3)--(2,3)--(2,2)--(1,2)--(1,1)--(0,1);
\draw[step=1em,gray](0,0) grid (7,7);
\draw[thickLine](0,0)--(0,1)--(0,2)--(0,3)--(0,4)--(0,5)--(1,5)--(2,5)--(3,5)--(3,6)--(4,6)--(5,6)--(6,6)--(6,7)--(7,7);
\draw (0,0)node[mydot]{}
(7,7)node[mydot]{};
\draw
(1,1) node[entries]{$1$}
(2,2) node[entries]{$2$}
(3,3) node[entries]{$3$}
(4,4) node[entries]{$4$}
(5,5) node[entries]{$5$}
(6,6) node[entries]{$6$}
(7,7) node[entries]{$7$}
;
\end{tikzpicture}
\begin{tikzpicture}[scale=0.9]
\fill[lightgray](0,0)--(0,1)--(0,2)--(0,3)--(0,4)--(0,5)--(0,6)--(1,6)--(2,6)--(3,6)--(4,6)--(5,6)--(6,6)--(6,7)--(7,7)--(6,7)--(6,6)--(5,6)--(5,5)--(4,5)--(4,4)--(3,4)--(3,3)--(2,3)--(2,2)--(1,2)--(1,1)--(0,1);
\draw[step=1em,gray](0,0) grid (7,7);
\draw[thickLine](0,0)--(0,1)--(0,2)--(0,3)--(0,4)--(0,5)--(0,6)--(1,6)--(2,6)--(3,6)--(4,6)--(5,6)--(6,6)--(6,7)--(7,7);
\draw (0,0)node[mydot]{}
(7,7)node[mydot]{};
\draw
(1,1) node[entries]{$1$}
(2,2) node[entries]{$2$}
(3,3) node[entries]{$3$}
(4,4) node[entries]{$4$}
(5,5) node[entries]{$5$}
(6,6) node[entries]{$6$}
(7,7) node[entries]{$7$}
;
\end{tikzpicture}
\begin{tikzpicture}[scale=0.9]
\fill[lightgray](0,0)--(0,1)--(0,2)--(0,3)--(1,3)--(2,3)--(3,3)--(3,4)--(3,5)--(3,6)--(3,7)--(4,7)--(5,7)--(6,7)--(7,7)--(6,7)--(6,6)--(5,6)--(5,5)--(4,5)--(4,4)--(3,4)--(3,3)--(2,3)--(2,2)--(1,2)--(1,1)--(0,1);
\draw[step=1em,gray](0,0) grid (7,7);
\draw[thickLine](0,0)--(0,1)--(0,2)--(0,3)--(1,3)--(2,3)--(3,3)--(3,4)--(3,5)--(3,6)--(3,7)--(4,7)--(5,7)--(6,7)--(7,7);
\draw (0,0)node[mydot]{}
(7,7)node[mydot]{};
\draw
(1,1) node[entries]{$1$}
(2,2) node[entries]{$2$}
(3,3) node[entries]{$3$}
(4,4) node[entries]{$4$}
(5,5) node[entries]{$5$}
(6,6) node[entries]{$6$}
(7,7) node[entries]{$7$}
;
\end{tikzpicture}
\]
\[
\begin{tikzpicture}[scale=0.9]
\fill[lightgray](0,0)--(0,1)--(0,2)--(0,3)--(0,4)--(1,4)--(2,4)--(3,4)--(3,5)--(3,6)--(3,7)--(4,7)--(5,7)--(6,7)--(7,7)--(6,7)--(6,6)--(5,6)--(5,5)--(4,5)--(4,4)--(3,4)--(3,3)--(2,3)--(2,2)--(1,2)--(1,1)--(0,1);
\draw[step=1em,gray](0,0) grid (7,7);
\draw[thickLine](0,0)--(0,1)--(0,2)--(0,3)--(0,4)--(1,4)--(2,4)--(3,4)--(3,5)--(3,6)--(3,7)--(4,7)--(5,7)--(6,7)--(7,7);
\draw (0,0)node[mydot]{}
(7,7)node[mydot]{};
\draw
(1,1) node[entries]{$1$}
(2,2) node[entries]{$2$}
(3,3) node[entries]{$3$}
(4,4) node[entries]{$4$}
(5,5) node[entries]{$5$}
(6,6) node[entries]{$6$}
(7,7) node[entries]{$7$}
;
\end{tikzpicture}
\begin{tikzpicture}[scale=0.9]
\fill[lightgray](0,0)--(0,1)--(0,2)--(0,3)--(0,4)--(0,5)--(1,5)--(2,5)--(3,5)--(3,6)--(3,7)--(4,7)--(5,7)--(6,7)--(7,7)--(6,7)--(6,6)--(5,6)--(5,5)--(4,5)--(4,4)--(3,4)--(3,3)--(2,3)--(2,2)--(1,2)--(1,1)--(0,1);
\draw[step=1em,gray](0,0) grid (7,7);
\draw[thickLine](0,0)--(0,1)--(0,2)--(0,3)--(0,4)--(0,5)--(1,5)--(2,5)--(3,5)--(3,6)--(3,7)--(4,7)--(5,7)--(6,7)--(7,7);
\draw (0,0)node[mydot]{}
(7,7)node[mydot]{};
\draw
(1,1) node[entries]{$1$}
(2,2) node[entries]{$2$}
(3,3) node[entries]{$3$}
(4,4) node[entries]{$4$}
(5,5) node[entries]{$5$}
(6,6) node[entries]{$6$}
(7,7) node[entries]{$7$}
;
\end{tikzpicture}
\begin{tikzpicture}[scale=0.9]
\fill[lightgray](0,0)--(0,1)--(0,2)--(0,3)--(0,4)--(0,5)--(0,6)--(1,6)--(2,6)--(3,6)--(3,7)--(4,7)--(5,7)--(6,7)--(7,7)--(6,7)--(6,6)--(5,6)--(5,5)--(4,5)--(4,4)--(3,4)--(3,3)--(2,3)--(2,2)--(1,2)--(1,1)--(0,1);
\draw[step=1em,gray](0,0) grid (7,7);
\draw[thickLine](0,0)--(0,1)--(0,2)--(0,3)--(0,4)--(0,5)--(0,6)--(1,6)--(2,6)--(3,6)--(3,7)--(4,7)--(5,7)--(6,7)--(7,7);
\draw (0,0)node[mydot]{}
(7,7)node[mydot]{};
\draw
(1,1) node[entries]{$1$}
(2,2) node[entries]{$2$}
(3,3) node[entries]{$3$}
(4,4) node[entries]{$4$}
(5,5) node[entries]{$5$}
(6,6) node[entries]{$6$}
(7,7) node[entries]{$7$}
;
\end{tikzpicture}
\begin{tikzpicture}[scale=0.9]
\fill[lightgray](0,0)--(0,1)--(0,2)--(0,3)--(0,4)--(0,5)--(0,6)--(0,7)--(1,7)--(2,7)--(3,7)--(4,7)--(5,7)--(6,7)--(7,7)--(6,7)--(6,6)--(5,6)--(5,5)--(4,5)--(4,4)--(3,4)--(3,3)--(2,3)--(2,2)--(1,2)--(1,1)--(0,1);
\draw[step=1em,gray](0,0) grid (7,7);
\draw[thickLine](0,0)--(0,1)--(0,2)--(0,3)--(0,4)--(0,5)--(0,6)--(0,7)--(1,7)--(2,7)--(3,7)--(4,7)--(5,7)--(6,7)--(7,7);
\draw (0,0)node[mydot]{}
(7,7)node[mydot]{};
\draw
(1,1) node[entries]{$1$}
(2,2) node[entries]{$2$}
(3,3) node[entries]{$3$}
(4,4) node[entries]{$4$}
(5,5) node[entries]{$5$}
(6,6) node[entries]{$6$}
(7,7) node[entries]{$7$}
;
\end{tikzpicture}
\]
\end{example}

\begin{theorem}
Let $r_{k}(n) \coloneqq \rrp_C(n,k)$. Then
\begin{equation}
 r_{k}(n+1) = \sum_{\substack{0\leq j \leq n \\ k \mid j}} r_{k}(j) r_{k}(n-j),\qquad r_k(0)\coloneqq 1.
\end{equation}
\end{theorem}
\begin{proof}
 Let $\pi$ be in $\RRP_C(n+1,k)$. Consider the last 1 in $f_\pi$. Then we have
 \begin{align*}
 f_\pi=\gamma_j,1,\gamma_{n-j}
 \end{align*}
 where $\gamma_j$ is the sub word having length $j$. 
 Clearly, $j$ is a multiple of $k$ and $0\leq j\leq n$. 
 Finally $\gamma_j\in\RRP_C(j,k)$ and $\gamma_{n-j}{-}1\in\RRP_C(n-j,k)$, since $\gamma_{n-j}$ starts with 2.
\end{proof}

In the same manner as the proof of \cref{cor:mainFussCatalan}, 
we have the following corollary.
\begin{corollary}\label{cor:subexFuss}
Let $m$ and $j$ be defined via $n=km+j$ with $0 \leq j < k$.
Then
\[
\rrp_C(n,k) =  \frac{j+1}{km+j+1} \binom{ (k+1)m + j }{km+j}
\]
and in particular
\[
\rrp_C(km,k) = \frac{1}{(k+1)m+1} \binom{(k+1)m+1}{m}
	= \frac{1}{km+1} \binom{(k+1)m}{m}.
\]
\end{corollary}

\section{Mod-k-alternating permutations avoiding two patterns}\label{sec:2patterns}

The systematic study of permutations avoiding two patterns of length 3,
was completed by Simion and Schmidt~\cite{SimionSchmidt1985}.
In this section we do the same for parity-alternating permutations 
and mod-$k$-alternating permutations.

Due to \cref{prop:onlyIdentityStronger},
the situation for $k \geq 3$ is in general simpler than the $k=2$ case.
In \cref{tab:2patterns}, we present an overview of our results in 
the $k=2$ case. We consider the sequence
$\{ \pap_{\sigma,\tau}(n) \}_{n\geq 1}$ and separate it into the cases $n=2m$ and $n=2m+1$.
Closed-form formulas for the 
sequences $\{\pap_{\sigma,\tau}(2m) \}_{m\geq 1}$ and $\{ \pap_{\sigma,\tau}(2m+1) \}_{m\geq 0}$
where $\sigma,\tau \in \{123,132,213,231,312,321\}$
are shown in the table.

The bold letters refer to cases which are treated below.
Note that \cref{lem:modkAltRevFlip,lem:modkInverse} can easily be generalized 
to sets of permutations avoiding two or more patterns,
e.g.~$\pap_{\sigma,\tau}(n,k) = \pap_{\revflip(\sigma),\revflip(\tau)}(n,k)$.
Cases marked with the same letter indicate situations which 
give the same enumeration after applying these lemmas.
We emphasize that the formulas are only valid for permutations of length at least $1$.

	\begin{table}[!ht]
		\begin{tabular}{lllllll}
			& \textbf{123} & \textbf{132} & \textbf{213} & \textbf{231} & \textbf{312} & \textbf{321}\\
			\midrule
			\textbf{123} &* &   &  & \\
			\textbf{132} & \textbf{A}; $2^{m-1}$, 1 & *  &  & \\
			 \textbf{213} & \textbf{A}; $2^{m-1}$, 1 &  \textbf{D}; $2^{m-1}$, $F_{2m+1}$ & * & \\
			\textbf{231} & \textbf{B}; $m$, $\binom{m}{2}{+}1$ &  \textbf{E}; $2^{m-1}$, $2^{m}$ & \textbf{E}; $2^{m-1}$, $2^{m}$ & *  &\\
			\textbf{312} & \textbf{B}; $m$, $\binom{m}{2}{+}1$ &  \textbf{E}; $2^{m-1}$, $2^{m}$ & \textbf{E}; $2^{m-1}$, $2^{m}$ & \textbf{F}; $F_{n}$ & * & \\
			\textbf{321} & \textbf{C}\textsuperscript{\textdagger}; 0 &   \textbf{G}; $\binom{m}{2}{+}1$ &  \textbf{G};  $\binom{m}{2}{+}1$ & \textbf{I}\textsuperscript{\ddag}; 1  & \textbf{I}; 1 & *\\
		\end{tabular}
		\caption{Pair of formulas in an entry represent the cases $n=2m$ and $n=2m+1$, respectively.
			A single formula covers to both cases.
		    The Fibonacci numbers are denoted by $F_n$, with $F_1=F_2=1$.}\label{tab:2patterns}
	\end{table}
\begin{remark}
	\textsuperscript{\textdagger}There are no permutations of length 5 or more that avoid $123$ and $321$ and \textsuperscript{\ddag}we only get the identity permutation that avoids $231$ and $321$  due to \cref{prop:onlyIdentityStronger} or (in the case $k=2$) a simple argument.
\end{remark}
Only the cases A, D and G are interesting for $k \geq 3$,
since otherwise, at least one pattern is covered by \cref{prop:onlyIdentityStronger},
and we then have $\rrp_{\sigma,\tau}(n,k)=1$ for all $n$.
The general formula for $\rrp_{\sigma,\tau}(n,k)$ in cases A, D and G
is given in the corresponding subsection.

\subsection{Case A (132,123)}
\begin{lemma}\label{lem:123-132general}
	Any permutation in $\symS_{132,123}(n)$ starts with at least $n-1$.
\end{lemma}
\begin{proof}
	Suppose $\pi\in\symS_{132,123}(n)$ and $\pi(1)<n-1$. Then $\pi$ contains either a $123$ pattern, $\pi(1)\;n{-1}\;n$, or a $132$ pattern, $\pi(1)\;n\;n{-1}$. Hence, $\pi(1)\geq n-1$.
\end{proof}

\begin{proposition}
For $n$ odd, $\PAP_{132,123}(n)$ consists only of the permutation
$[n,n-1,n-2,\dotsc,2,1]$. 
When $n=2m$, we have that $\pap_{132,123}(2m) = 2^{m-1}$.
\end{proposition}
\begin{proof}
Suppose $\pi \in \PAP_{132,123}(n)$, for odd $n$. From \cref{lem:123-132general}, $\pi$ starts with $n$ and then the rest of the permutation must follow the same rule inductively. Hence, $\pi$ is the unique permutation $[n,n-1,n-2,\dotsc,2,1]$.

\medskip 

For the second claim, suppose that
\[
  \pi = [ \alpha, 2m, \beta ] \in \PAP_{132,123}(2m).
\]
Then it is easy to verify that the permutations 
\begin{equation}\label{eq:caseAperms}
 [2m+1, 2m+2, \alpha, 2m, \beta ] \text{ and }
 [2m+1, 2m,   \alpha, 2m+2, \beta ] 
\end{equation}
are elements in $\PAP_{132,123}(2m+2)$.

Now, for any $\pi' \in \PAP_{132,123}(2m+2)$, we have that $2m+1$ is in the very first position, from \cref{lem:123-132general}. In a similar manner, we must have that $\pi'(2) \in \{ 2m, 2m+2\}$.
We have now showed that all elements in $\PAP_{132,123}(2m+2)$
are in one of the forms described in \eqref{eq:caseAperms}.
By an inductive argument, it follows that $\pap_{132,123}(2m) = 2^{m-1}$.
\end{proof}

\begin{proposition}
Let $k\geq 3$ and write $n = km+l$, where $0 \leq l < k$.
We then have that
	\[
	\rrp_{132,123}(km+l,k)=
	\begin{cases}
	  1 &\text{ if $(k,l)\in\{(3,1),(3,2),(4,2)\}$}, \\
	  0 & \text{ otherwise}.
	\end{cases}
\]
\end{proposition}
\begin{proof}
Suppose $\pi\in \RRP_{132,123}(n,k)$ and consider the position of $1$ in $\pi$. 
Since $\pi$ avoids both $132$ and $123$, we must have that $1$ is in the last, or penultimate position.
Consequently, $l \in \{1,2\}$, and we shall treat these 
two cases separately.

\medskip 
\noindent 
\textbf{Case $l=1$:}
Since $k \geq 3$, $2$ must appear to the left of $1$, with at least some entry between.
Moreover, since $\pi$ must avoid both $132$ and $123$, there cannot be 
more than one entry between $2$ and $1$, which implies that $k=3$.
The same reasoning as above shows that $3$ must be placed between $2$ and $1$,
so that $\pi$ ends with $231$.
We may now treat the first $n-3$ entries in $\pi$ as an element in $\RRP_{132,123}(n-3,k)$,
and by induction conclude that $\pi$ is the unique permutation
\[
\pi=[3m{+}1, \,\,\, 3m{-}1,3m,3m{-}2, \,\, \dotsc, \,\,8,9,7, \,\, 5,6,4, \,\,2,3,1],
\]
where the largest entry is first, and the remaining entries appear in blocks of $3$.

\medskip 
\noindent 
\textbf{Case $l=2$:} It is straightforward to show that 
we must have $2$ at the last position of $\pi$, in order 
to avoid the two patterns.
Now, a similar argument as above shows that 
\[
  \pi = [ \dotsc, 3,1,2] \text{ or } \pi = [ \dotsc, 3,j,1,2]
\]
where $4 \leq j \leq n$. The first option gives $k=3$,
while the second option forces $k=4$ and $j=4$.

In case $k=3$, we have that $\pi$ ends with $3,1,2$,
so an inductive argument then gives
\[
 \pi = [3m{+}1, 3m{+}2, \,\,\, 3m,3m{-}2,3m{-}1, \,\, \dotsc, \,\,  9,7,8, \,\,  6,4,5, \,\,3,1,2].
\]
Similarly, the case $k=4$ gives 
\[
 \pi = [4m{+}1, 4m{+}2, \,\,\, 4m{-}1,4m,4m{-}3,4m{-}2, \,\, \dotsc, \,\,  7,8,5,6, \,\,3,4,1,2].
\]
\end{proof}

\subsection{Case B (231,123)}

For $n$ odd, we have $\pap_{321,132}(n)=\pap_{231,123}(n)$
because the reversal map (\cref{lem:revBij}) gives a bijection between the sets.
The latter set is enumerated further down in Case~G.

The case of even $n$ is covered by the following lemma.
\begin{lemma}
Suppose $\pi \in \PAP_{231,123}(2m)$.
Then
\[
  \pi = [k, k-1, k-2,\dotsc,1,2m,2m-1,\dotsc,k+1]
\]
for some odd $k$. Consequently, $\pap_{231,123}(2m) = m$.
\end{lemma}
\begin{proof}
Because $\pi$ is $231$-avoiding, all entries to the right of $2m$ must be larger than
all entries to its left. Moreover, it is then straightforward to show that
both these intervals must be decreasing in order for $\pi$ to avoid the pattern $123$.
There are exactly $m$ choices of the first entry in $\pi$---any odd number in $[2m]$.
From these observations, the lemma follows.
\end{proof}

\subsection{Case D (213,132)}

\begin{proposition}
For $k \geq 1$ and $0 \leq l <k$, we have that
\[
\rrp_{213,132}(km+l,k) =
\begin{cases}
2^{m-1} &\text{ if $l=0$}, \\
F_{2m+1}& \text{ otherwise}.
\end{cases}
\]
\end{proposition}
\begin{proof}
For $k=1$, this reduces to proving that $\symS_{213,132}(m)=2^{m-1}$.
This has been done already by R.~Simion and F.~Schmidt~\cite{SimionSchmidt1985}. 
In the remaining part of the proof, we treat the case $k\geq 2$.
\medskip

\noindent
\textbf{Case $l \geq 1$}. 
From the recursive formula $F_{2m+1}=F_{2m}+F_{2m-1}$ of the Fibonacci numbers,
we can easily deduce the following recursive formula 
involving only the odd indexed Fibonacci numbers:
\begin{align}\label{eq:oddFibonacci}
	F_{2m+1}=2F_{2(m-1)+1}+F_{2(m-2)+1}+F_{2(m-3)+1}+\dotsb + F_{3}+F_1.
\end{align}
Our goal is to show the same recursive structure for elements in $\RRP_{213,132}(km+l,k)$.
Suppose $\pi \in \RRP_{213,132}(km+l,k)$. 
We claim that $\pi$ can be obtained in one of the following three ways.
\begin{enumerate}
\item[(1a)] From $\pi' \in \RRP_{213,132}(k(m-1)+l,k)$
		 by inserting the consecutive integers
		\[
		k(m{-}1){+}l{+}1,\,\,k(m{-}1){+}l{+}2,\dotsc, km{+}l{-}1, \,\,km{+}l
		\]
	 in this order \emph{immediately after} the largest entry in $\pi'$.
	 
	 \item[(1b)] 
	 From $\pi' \in \RRP_{213,132}(k(m-1)+l,k)$, by inserting
	 the two sequences of consecutive integers
	 \[
	 k(m{-}1){+}l{+}i,\,\,k(m{-}1){+}l{+}i{+}1,\dotsc, km{+}l,
	 \]
	   and then
	   \[
	   k(m{-}1){+}l{+}1,\dotsc, k(m{-}1){+}l{+}i{-}1,
	   \]
	  in this order, for some $1<i\leq k$ such 
	  that $k(m{-}1){+}l{+}i\equiv 1\mod k$, in the beginning of $\pi'$---
	  in fact, $i=k-l+1$.

\item[(2)] From $\pi''\in\RRP_{213,132}(k s +l,k)$ for any $1\leq s< m-1$, by inserting
the two sequences of consecutive integers
 \[
 k(m{-}1){+}l{+}i,\,\,k(m{-}1){+}l{+}i{+}1,\dotsc, km{+}l,
 \]
 and then
 \[
 ks{+}l{+}1,\,\,ks{+}l{+}2,\dotsc, k(m{-}1){+}l{+}i{-}1,
 \]
in the beginning of $\pi''$, where $i$ satisfies the same condition as in the previous case.
\end{enumerate}
It is straightforward to verify that the three cases 
are mutually exclusive, and are indeed elements of $\RRP_{213,132}(km+l,k)$.
We must show that every element in $\RRP_{213,132}(km+l,k)$ 
belongs to one of these cases.

Consider the largest element $km+l$ in $\pi$. 
All elements to its left must appear in increasing order. 
Moreover, there cannot be any gaps, as that 
would produce a $132$-pattern. There are two cases to consider;
either $km+l$ occurs after $k(m-1)+l$ and there are exactly $k-1$ elements, all are greater than $k(m-1)+l$, between them in increasing order. This implies that $\pi$ belongs to case (1a) above.
Otherwise, the entries to the left of $km+l$
form the interval
\[
\underbrace{k(m{-}1){+}l{+}i,\,\,k(m{-}1){+}l{+}i{+}1,\dotsc, km{+}l-1}_{k-i},km{+}l,
\]
for unique $i \in [k]$. 
The largest entry in $\pi$, to the right of $km+l$,
is $k(m{-}1){+}l{+}i-1$, so 
the entries between $km+l$ and $k(m{-}1){+}l{+}i-1$ must form 
an increasing sequence of consecutive integers, the first one being 
of the form $ks+l+1$ for some $s \in [m-1]$.
This puts $\pi$ in either case (1b) (for $s=m-1$) or case (2).

This shows that we have the recursion 
 \begin{align*}
 \rrp_{213,132}(km+l,k) &= 2 \, \rrp_{213,132}(k(m-1)+l,k)+\rrp_{213,132}(k(m-2)+l,k)+ \\
  &\dotsb +\rrp_{213,132}(k+l,k)+\rrp_{213,132}(l,k).
\end{align*}
Thus, (after checking initial conditions) we have that
$\{ \rrp_{213,132}(km+l,k) \}_{m=0}^\infty$ satisfies 
the same recursion as in \eqref{eq:oddFibonacci}. 
Therefore, $\rrp_{213,132}(km+l,k)=F_{2m+1}$.

\bigskip 
\noindent
\textbf{Case $l=0$}. 
Suppose that $\pi\in\RRP_{213,132}(k(m+1),k)$. 
By a similar argument as above, $\pi$ either starts with
\[
km{+}1,\,\,km{+}2,\dotsc, k(m{+}1){-}1,\,\, k(m{+}1),
\]
or $km$ is followed by this sequence of integers.
Thus, we can obtain $\pi$ from $\pi'\in \RRP_{213,132}(km,k)$ by
inserting
\[
km{+}1,\,\,km{+}2,\dotsc, k(m{+}1){-}1,\,\, k(m{+}1)
\]
either immediately after $km$, or at the beginning of $\pi'$. 
Hence,
\begin{align*}
	\rrp_{213,132}(k(m+1),k)=2\,\rrp_{213,132}(km,k),
\end{align*}
and we have the initial condition $\rrp_{213,132}(k,k)=1$ (the identity permutation),
so it follows that $\rrp_{213,132}(km,k)=2^{m-1}$.
\end{proof}

\begin{corollary}
We have 
\[
	\pap_{213,132}(2m) =2^{m-1}\qquad \text{ and }\qquad \pap_{213,132}(2m+1)=F_{2m+1}.
\]
\end{corollary}

\subsection{Case E (231,132)}
\begin{proposition}
	We have that
	\[
	\pap_{231,132}(2m)= 2^{m-1}\quad\text{ and }\quad
	\pap_{231,132}(2m+1)=2^m .
	\]
\end{proposition}
\begin{proof}
	If $\pi \in \symS_{231,132}(n)$, it is clear that $n$ must be either in the
	very beginning, or at the very end of $\pi$. 
	Thus, $\pi\in\PAP_{231,132}(2m+1)$ is obtained from $\pi'\in \PAP_{231,132}(2m-1)$ 
	by inserting $2m,2m{+}1$ at the end of $\pi'$, or by inserting $2m{+}1,2m$ 
	at the beginning of $\pi'$. 
	Moreover, every $\pi\in\PAP_{231,132}(2m)$ can be obtained from some 
	$\pi'\in \PAP_{231,132}(2m-1)$ by inserting $2m$ at the end of $\pi'$.  
	Hence,
	\[
	\pap_{231,132}(2m+1)=2\, \pap_{231,132}(2m-1)\quad\text{ and }\quad \pap_{231,132}(2m)=\pap_{231,132}(2m-1),
	\]
	where $\pap_{231,132}(1)=1$ and $\pap_{231,132}(2)=1$. 
	By a simple inductive argument, the statement follows.
\end{proof}

\subsection{Case F, (312,231)}

For $k \geq 3$, $\RRP_{312,231}(n,k)$ is just the identity permutation,
due to \cref{prop:onlyIdentityStronger}.

For $k=2$, we have the following proposition.
\begin{proposition}
We have that $\pap_{312,231}(n) = F_n$, the $n^\thsup$ Fibonacci number.
\end{proposition}
\begin{proof}
It is easy to verify that this holds for $n \leq 3$,
so suppose that $\pi\in \PAP_{312,231}(n)$ for some $n \geq 4$.

If $\pi(n)\neq n$, then $n-1$ must appear immediately to the right of 
$n$, in order to avoid $312$. 
Similarly, $n-2$ cannot be anywhere left of $n$,
since then $n-2$, $n$ and $\pi(n)$
would form a $231$-pattern.
Now, if there was some element between $n-1$ and $n-2$,
there would be a $312$-pattern in $\pi$. 
Hence, $n,n-1,n-2$ must be a subword of $\pi$.

It follows that any $\pi\in \PAP_{312,231}(n)$ can be either constructed 
from some $\pi' \in \PAP_{312,231}(n-1)$, by appending $n$
at the very end, or by inserting $n,n-1$ immediately before $n-2$
in some $\pi'' \in \PAP_{312,231}(n-2)$.
Thus, we have the classical Fibonacci recursion and the statement follows.
\end{proof}
Observe that for odd $n$, we may apply \cref{lem:revBij}
and then see that $\pap_{312,231}(n) = \pap_{132,213}(n)$,
where the latter is $F_n$ according to Case~D.

\subsection{Case G, (321,132)}

\begin{lemma}\label{lem:cons}
Let $\pi \in \RRP_{321,132}(km+l,k)$, where $0\leq l<k$ and $k \geq 2$. 
Then for any $j \in [m]$, the entries
\[
k(j{-}1){+}1,\,\, k(j{-}1){+}2,\,\, \dotsc,\,\, kj{-}1,\,\, kj
\]
form a subword of $\pi$. 

Moreover, the last $l$ numbers in $[km+l]$ 
must appear in increasing order at the very end of $\pi$.
\end{lemma}
\begin{proof}
	Suppose $k(j-1)+i$ is on the left of $k(j-1)+i-1$, for some $1<i\leq k$. 
	Then $\pi(1), k(j{-}1){+}i, k(j{-}1){+}i{-}1$ form either a
	$132$- or $321$-pattern, depending on the
	value of $\pi(1)$ (we know $\pi(1)\neq k(j-1)+i$ since $\pi(1) \equiv 1 \mod k$). 
	Therefore,
	\[
	k(j{-}1){+}1,\,\, k(j{-}1){+}2,\,\, \dotsc,\,\, kj{-}1,\,\, kj
	\]
	appear in this order.
	
	Now suppose that $k(j-1)+i-1$ and $k(j-1)+i$, for some $1<i\leq k$, 
	are not adjacent. Then there are at least $k$
	entries between them since $\pi$ is a mod-$k$-alternating permutation. 
	Each such entry must be
	less than $k(j-1)+i-1$, otherwise, we have a $132$-pattern.
	Now consider the position of $1$ in $\pi$. If
	$1$ appears to the left of $k(j-1)+i-1$, then
	$1,\dotsc,\,\,k(j{-}1){+}i{-}1,\,\,t$ form a $132$-pattern,
	where $t$ is the number succeeding $k(j-1)+i-1$. 
	If $1$ is on the right, it can not be immediately after since $i>1$, of
	$k(j-1)+i-1$, then $t\neq 1$ and $k(j{-}1){+}i{-}1,\,\,t,\dotsc,\,\,1$ form a
	$321$-pattern.
	
	Now, it remain to show the last statement in the lemma.
	For $l=0$, there is noting more to prove. 
	If $l>0$ we can argue in the same manner as above, that $km+1,km+2,\dotsc,km+l$
	must form a subword of $\pi$. Since all other entries are partitioned into
	blocks of length $k$, we must have that the $l$ largest entries
	appear at the very end. Otherwise, the entry immediately after $km+l$
	would have the wrong remainder mod $k$.
\end{proof}

\begin{proposition}
For any $k \geq 1$, $m \geq 0$ and $0\leq l<k$, we have that
\[
\rrp_{321,132}(km+l,k)=\binom{m}{2}+1.
\]
\end{proposition}
\begin{proof}
Let $\pi \in \RRP_{321,132}(km+l,k)$. By the second statement in \cref{lem:cons},
we can simply disregard the last $l$ entries and conclude that 
$\rrp_{321,132}(km+l,k) = \rrp_{321,132}(km,k)$.
By the first property in \cref{lem:cons}, we note that there is a simple bijection
from $\RRP_{321,132}(km,k)$ to $\symS_{321,132}(m)$.
We simply keep only the first entry in each block of $k$ letters,
and standardize the result.
Thus, $\rrp_{321,132}(km+l,k)=|\symS_{321,132}(m)|$.
Finally, it has already been shown that $|\symS_{321,132}(m)| = \binom{m}{2}+1$, 
see \cite[Prop. 11]{SimionSchmidt1985}.
\end{proof}

\begin{corollary}
We have that
\[
	\pap_{321,132}(2m)=\pap_{321,132}(2m+1)=\binom{m}{2}+1.
\]
\end{corollary}

\section{Remarks}

The attentive reader has most likely noticed that 
we have said nothing about parity-alternating permutations
avoiding the single pattern 123, or the single pattern 321.
The sequence with the number of 123-avoiding PAPs of size $n$,
starts as
\[
 1, 1, 1, 3, 3, 10, 11, 37, 44, 146, 185, 603, 808, 2576,\dotsc
\]
Similarly, for $321$-avoiding PAPs, we get 
\[
 1, 1, 1, 2, 3, 6, 11, 22, 44, 89, 185, 382, 808, 1702, \dotsc.
\]
Recall from \cref{lem:revBij} that the reversal map 
is a bijection between the sets 
$\PAP_{321}(2m+1)$ and $\PAP_{123}(2m+1)$,
so $\pap_{321}(2m+1)=\pap_{123}(2m+1)$ for all $m\geq 0$.
We have not managed to produce recursive formulas 
for these sequences. The numbers above were obtained by 
brute-force enumeration,
and we note that there is no hit in the OEIS~\cite{OEIS}.

However, we have managed to find a property of $321$-avoiding PAPs,
\cref{prop:321AvoPAPs}, which might exploited to produce a recursive formula.

\begin{definition}
	A \defin{left-to-right maximum} of $\pi\in\symS_n$ is an integer $i\in[n]$ such that $i>j$, 
	for all $j\in[n]$ to the left of $i$. 
	We call the integers that are not left-to-right maxima as \defin{hole values}.
\end{definition}

\begin{proposition}\label{prop:321AvoPAPs}
	Let $\pi\in\PAP_{321}(n)$ be non-identity. Then
	\begin{enumerate}
		\item there exists a pair of hole values in $\pi$ that are consecutive and
		\item the highest such pair of numbers either stands in 
		adjacent positions or have an even length sequence of 
		consecutive left-to-right-maxima between the numbers.
	\end{enumerate}
\end{proposition}
\begin{proof}
	Let $i\in[1,n-1]$ be the highest position such that no exceedance 
	in $\pi\in\PAP_{321}(n)$ till $i$. 
	Then $\pi_j=j$, for all $j=1,2,\dotsc, i{-}1$, and $\pi_{i+1}>i{+}1$ 
	since $\pi_i>i$ and $\pi$ is a $321$-avoiding PAP. 
	Thus, $i$ and $i+1$ are hole values, which shows the existence.
	
	\medskip
	
	 Now consider the highest such pair of consecutive hole values in $\pi$, 
	 say $h$ and $h+1$. The integer $h+1$ is always from the right of $h$. 
	 Otherwise, $\pi_h, h{+}1, h$ forms a $321$ pattern. If $h$ and $h+1$ are 
	 adjacent, then we are done. Otherwise, there are even number, 
	 at least two, of positions between $h$ and $h+1$ since $\pi$ is a PAP. 
	 Moreover, the numbers between them are left-to-right maxima 
	 since $\pi$ is $321$-avoiding. If the numbers in between are consecutive, 
	 then we are done. Otherwise, we have a pair of consecutive hole values 
	 which are greater than $h$ and $h+1$. This is a contradiction since $h$ and $h+1$ are the highest.
\end{proof}

\section*{Acknowledgements}

Thanks to Ran Pan for proposing the problem about pattern avoidance of 
parity-alternating permutations.\url{https://mathweb.ucsd.edu/~projectp/problems/p10.html}

We appreciate the hospitality we get from
department of mathematics, Stockholm University. 
We specially acknowledge the useful discussions with Jörgen Backelin from Stockholm University.
 
The second and third authors are grateful for the 
financial support extended by the cooperation 
agreement between International Science Program (ISP) 
at Uppsala University and Addis Ababa University.
We are also grateful for the financial support by the Wennergren foundation.

Dun Qiu is supported by NSFC grant No. 12001037.

The On-line Encyclopedia of integer sequences (OEIS)
has been a valuable tool in this research project.

\begin{appendices}

\section{Appendix: Lagrange inversion}\label{apx:lagrange}

In this section, we prove a version of Lagrange inversion,
where the formal power series do not necessarily have integer powers.
Our proof follows the one outlined in \cite[Ch. 5.4]{StanleyEC2},
and lecture notes by P.~Maygar.

Let $d$ be a fixed positive integer and let $\setC(\!(x^{1/d})\!)$ 
be the field of Laurent series of the form
\[
 \sum_{k \geq -N/d} a_{k/d} x^{k/d}, \qquad a_{k/d} \in \setC, \qquad N \in \setZ.
\]
Moreover, we let $\setC[\![x^{1/d}]\!]$ be such Laurent series with only non-negative powers
appearing, and $x\setC[\![x^{1/d}]\!]$ be such series where all powers appearing are at least $1$.

\begin{lemma}\label{lem:derivCoeff}
Let $h \in \setC(\!(x^{1/d})\!)$. Then
\begin{enumerate}
 \item $h'(x) \vert_{x^{-1}} = 0$.
\item 
For $f \in x\setC[\![x^{1/d}]\!]$, with $f(x) \vert_{x^1} \neq 0$
and any $j \in \setQ$, we have
\[
 f(x)^j f'(x) \vert_{x^{-1}} = 
 \begin{cases}
 1 & \text{ if $j=-1$} \\
 0 & \text{ otherwise.}
 \end{cases}
\]
\end{enumerate}
\end{lemma}
\begin{proof}
The first statement is easy. Now if $j \neq -1$, taking $h(x) = f(x)^{j+1}$
gives that $h'(x) = (j+1)f(x)^j f'(x)$ and thus $f(x)^j f'(x) \vert_{x^{-1}} = 0$. 
For $j=-1$, and $f(x) = ax + b x^{1+1/d}+\dotsb$
\[
 \left. \frac{f'(x)}{f(x)} \right\vert_{x^{-1}} = \left. \frac{x \cdot f'(x)}{f(x)} \right\vert_{x^{0}}
 =
 \left. x \frac{a + bx^{1/d}+\dotsb}{ax + b x^{1+1/d}+\dotsb} 
 \right\vert_{x^{0}}
  = \frac{a}{a} =1.
\]
\end{proof}

\begin{theorem}[Lagrange inversion]
Let $f,g \in x\setC[\![x^{1/d}]\!]$ be inverses. Then for any $r \in \setQ$,
\begin{equation}\label{eq:lagrange1}
 g(x) \vert_{x^{r}} = \frac1r \cdot f(t)^{-r} \vert_{t^{-1}}.
\end{equation}
In particular, if $f(x) = x/ \phi(x)$ and $g(x) = x \cdot \phi( g(x) )$
for some $\phi \in \setC[\![x^{1/d}]\!]$ with non-vanishing constant coefficient,
then
\begin{equation}\label{eq:lagrange2}
 g(x) \vert_{x^{r}} = \frac1r \cdot \phi(t)^{r} \vert_{t^{r-1}}.
\end{equation}
\end{theorem}
\begin{proof}
We let $g(x) = \sum_{j \geq 1} c_j x^j$, where $j$ ranges over positive integer multiples of $1/d$.
Now,
\[
  x = g(f(x)) = \sum_{j \geq 1} c_j f(x)^j.
\]
Taking derivatives and then dividing by $f(x)^r$ gives
\[
 f(x)^{-r} = \sum_{j \geq 1} j c_j f(x)^{j-r-1} f'(x).
\]
We now take the coefficient of $x^{-1}$ on both sides. 
By \cref{lem:derivCoeff}, we get 
\[
 f(x)^{-r} \big \vert_{x^{-1}} = r \cdot c_r,
\]
and this is precisely \eqref{eq:lagrange1}. 
The statement in \eqref{eq:lagrange2} is now straightforward to prove as well.
\end{proof}
Note that since $\phi(x)$ in \eqref{eq:lagrange2} has a non-zero constant term,
we can find the formal power series of $\phi(x)^r$ by using 
the binomial power series 
\[
 (1+x)^{r} = \sum_{n \geq 0} \binom{r}{n} x^n = 
 \sum_{n \geq 0} \frac{r(r-1)\dotsc(r-n+1)}{n!} x^n.
\]
\end{appendices}

\bibliographystyle{alphaurl}
\bibliography{bibliography}

\end{document}